\documentclass{amsart}
\usepackage[all]{xy}
\usepackage{color}
\usepackage{stmaryrd}
\usepackage{amsthm}
\usepackage{amssymb}
\usepackage[colorlinks=true]{hyperref}

\numberwithin{equation}{subsection}

\newtheorem{theorem}{Theorem}[section]
\newtheorem*{theorem*}{Theorem}
\newtheorem{lemma}[theorem]{Lemma}
\newtheorem{sublemma}[theorem]{Sublemma}
\newtheorem{proposition}[theorem]{Proposition}
\newtheorem{corollary}[theorem]{Corollary}
\newtheorem*{corollary*}{Corollary}

\theoremstyle{remark}
\newtheorem{definition}[theorem]{Definition}
\theoremstyle{remark}
\newtheorem{example}[theorem]{Example}

\theoremstyle{remark}
\newtheorem{remark}[theorem]{Remark}
\theoremstyle{remark}
\newtheorem{notation}[theorem]{Notation}
\DeclareMathOperator{\Mot}{Mot}

\DeclareMathOperator{\incl}{incl}
\DeclareMathOperator{\id}{id}

\DeclareMathOperator{\NNum}{NNum} 
\DeclareMathOperator{\AM}{AM} 
\DeclareMathOperator{\MAM}{MAM} 
\DeclareMathOperator{\NAM}{NAM} 
\DeclareMathOperator{\NMAM}{NMAM} 
\DeclareMathOperator{\NMix}{NMix} 
\DeclareMathOperator{\DM}{DM} 

\newcommand{\Ho}{\mathrm{Ho}}
\newcommand{\ko}{\: , \;}
\newcommand{\too}{\longrightarrow}
\newcommand{\dg}{\mathrm{dg}}
\newcommand{\dgHo}{\mathrm{H}}

\newcommand{\cA}{{\mathcal A}}
\newcommand{\cB}{{\mathcal B}}
\newcommand{\cC}{{\mathcal C}}
\newcommand{\cD}{{\mathcal D}}
\newcommand{\cE}{{\mathcal E}}
\newcommand{\cF}{{\mathcal F}}
\newcommand{\cG}{{\mathcal G}}
\newcommand{\cH}{{\mathcal H}}

\newcommand{\cS}{{\mathcal S}}

\newcommand{\bbA}{\mathbb{A}}

\newcommand{\bbC}{\mathbb{C}}

\newcommand{\bbF}{\mathbb{F}}

\newcommand{\bbQ}{\mathbb{Q}}
\newcommand{\bbP}{\mathbb{P}}

\newcommand{\bbZ}{\mathbb{Z}}

\newcommand{\op}{\mathrm{op}} 
\newcommand{\ie}{\textsl{i.e.}\ }

\newcommand{\Hmo}{\mathrm{Hmo}}

\newcommand{\perf}{\mathrm{perf}} 

\newcommand{\add}{\mathrm{add}}
\newcommand{\Hom}{\mathrm{Hom}} 
\newcommand{\rep}{\mathrm{rep}} 
\newcommand{\Fun}{\mathrm{Fun}} 
\newcommand{\dgcat}{\mathrm{dgcat}}

\newcommand{\Spt}{\mathrm{Spt}}

\newcommand{\uHom}{\underline{\mathrm{Hom}}}

\newcommand{\dgS}{\cS}
\newcommand{\dgD}{\cD}

\newcommand{\internalcomment}[1]{}

\title[NC mixed (Artin) motives and their motivic Hopf dg algebras]{Noncommutative mixed (artin) motives \\and their motivic Hopf dg algebras}
\author{Gon{\c c}alo~Tabuada}

\address{Gon{\c c}alo Tabuada, Department of Mathematics, MIT, Cambridge, MA 02139, USA}
\email{tabuada@math.mit.edu}
\urladdr{http://math.mit.edu/~tabuada}

\date{\today}


\thanks{The author was partially supported by a NSF CAREER Award.}

\begin{document}
\begin{abstract}
This article is the sequel to \cite{Artin}. We start by developing a theory of noncommutative (=NC) mixed motives with coefficients in any commutative ring. In particular, we construct a symmetric monoidal triangulated category of NC mixed motives, over a base field $k$, and a full subcategory of NC mixed Artin motives. Making use of Hochschild homology, we then apply Ayoub's weak Tannakian formalism to these motivic categories. In the case of NC mixed motives, we obtain a {\em motivic Hopf dg algebra}, which we describe explicitly in terms of Hochschild homology and complexes of exact cubes. In the case of NC mixed Artin motives, we compute the associated Hopf dg algebra using solely the classical category of mixed Artin-Tate motives. Finally, we establish a short exact sequence relating the Hopf algebra of continuous functions on the absolute Galois group with the motivic Hopf dg algebras of the base field $k$ and of its algebraic closure. Along the way, we describe the behavior of Ayoub's weak Tannakian formalism with respect to orbit categories and relate the category of NC mixed motives with Voevodsky's category of mixed motives.
\end{abstract}
\maketitle

\vskip-\baselineskip
\vskip-\baselineskip
\section{Introduction}
\subsection*{Dg categories}
A {\em dg category $\cA$}, over a base field $k$, is a category enriched over complexes of $k$-vector spaces; see \S\ref{sec:dg}. Every (dg) $k$-algebra $A$ gives naturally rise to a dg category with a single object. Another source of examples is provided by schemes since the category of perfect complexes $\perf(X)$ of every quasi-compact quasi-separated $k$-scheme $X$ admits a canonical dg enhancement $\perf_\dg(X)$. Following Kontsevich \cite{ENS}, a dg category $\cA$ is called {\em smooth} if it is compact as a bimodule over itself and {\em proper} if $\sum_i \mathrm{dim}\, \dgHo^i \cA(x,y) < \infty$ for any pair $(x,y)$ of objects. Examples include the finite dimensional $k$-algebras of finite global dimension (when $k$ is perfect) and the dg categories $\perf_\dg(X)$ associated to smooth projective $k$-schemes $X$. In what follows, $\dgcat(k)$ denotes the category of (small) dg categories and $\dgcat_{\mathrm{sp}}(k)$ the full subcategory of smooth proper dg categories.
\subsection*{Noncommutative Artin motives}
Let $R$ be a commutative ring of coefficients. Recall from \cite{Semi,Kontsevich,Galois} the construction of the category of noncommutative numerical motives $\NNum(k)_R$ and of the symmetric monoidal functor $\mathrm{U}_R: \dgcat_{\mathrm{sp}}(k) \to \NNum(k)_R$. When $R$ is a $\bbQ$-algebra, there exists an $R$-linear symmetric monoidal functor $\psi$ making the following diagram commute
\begin{equation}\label{eq:diagram-pure}
\xymatrix{
\mathrm{SmProj}(k)^\op \ar[d]_-{\mathfrak{h}_R} \ar[rrr]^-{X \mapsto \perf_\dg(X)} &&& \dgcat_{\mathrm{sp}}(k) \ar[d]^-{\mathrm{U}_R} \\
\mathrm{Num}(k)_R \ar[rrr]_-{\psi} &&& \NNum(k)_R\,,
}
\end{equation}
where $\mathrm{SmProj}(k)$ stands for the category of smooth projective $k$-schemes and $\mathrm{Num}(k)_R$ for the category of numerical motives. The category of {\em noncommutative Artin motives $\NAM(k)_R$} was introduced in \cite{Artin} as the smallest additive idempotent complete full subcategory of $\NNum(k)_R$ containing the objects $\mathrm{U}_R(l)$, where $l/k$ is a finite separable field extension. As proved in {\em loc. cit.}, the induced functor $\psi: \mathrm{AM}(k)_R \to \NAM(k)_R$, from the classical category of Artin motives to the category of noncommutative Artin motives, is an equivalence of categories.
\subsection*{Motivic Galois groups}
As proved in \cite{Semi}, the category $\NNum(k)_R$ is abelian semi-simple. Assuming the noncommutative standard conjecture $C_{NC}$ (= K{\"u}nneth) and that $k$ is of characteristic zero, $\NNum(k)_R$ can be made into a Tannakian category $\NNum^\dagger(k)_R$ by modification of the symmetry isomorphism constraints. Moreover, assuming the noncommutative standard conjecture $D_{NC}$ (= homological equals numerical) and that $R$ is a $k$-algebra, $\NNum^\dagger(k)_R$ becomes a neutral Tannakian category. Consequently, we obtain a {\em motivic Galois group} $\mathrm{Gal}(\NNum^\dagger(k)_R)$; see \cite{Galois}. As proved in \cite{Artin}, $\NAM^\dagger(k)_R=\NAM(k)_R$ and when $k\subset \bbC$ and $R$ is a $\bbC$-algebra, the associated affine group scheme $\mathrm{Gal}(\NAM(k)_R)$ identifies with the absolute Galois group $\mathrm{Gal}(\overline{k}/k)$. Moreover, we have the  short exact sequence
\begin{eqnarray}\label{eq:short-exact}
\quad\quad\quad  1 \too \mathrm{Gal}(\NNum^\dagger(\overline{k})_R) \too \mathrm{Gal}(\NNum^\dagger(k)_R) \too \mathrm{Gal}(\overline{k}/k)\too 1\,,
\end{eqnarray}
where the first map is induced by the inclusion $\NAM(k)_R \subset \NNum^\dagger(k)_R$ and the second one by the base-change functor $-\otimes_k \overline{k}: \NNum^\dagger(k)_R \to \NNum^\dagger(\overline{k})_R$.

\smallskip

The purpose of this article is to develop the mixed analogues of the aforementioned constructions and results. This will require new arguments and new mathematical tools which are of independent interest. Since the mixed world is much richer than the pure work, interesting new phenomena will occur; consult for instance Proposition \ref{prop:comparison} and Theorem \ref{thm:main1} below.
\section{Statement of results}\label{sec:statements}
\subsection*{Noncommutative mixed Artin motives}
Recall from \S\ref{sec:NCMixed} the construction of the symmetric monoidal triangulated categories of {\em noncommutative mixed motives} $\NMix(k;R)$ and $\NMix(k;R)^\oplus$, and of the symmetric monoidal functor 
$$U_R: \dgcat_{\mathrm{sp}}(k)\too \NMix(k;R) \subset \NMix(k;R)^\oplus\,.$$
Let $\mathrm{DM}(k;R)$ be Voevodsky's category of motives and $M_R: \mathrm{SmProj}(k) \to \mathrm{DM}(k;R)$ the associated symmetric monoidal functor. The above bridge \eqref{eq:diagram-pure}, relating the commutative with the  noncommutative world, admits the following mixed analogue:
\begin{theorem}\label{thm:main0}
When $R$ is a $\bbQ$-algebra, there exists an $R$-linear triangulated symmetric monoidal functor $\Psi$ making the following diagram commute
\begin{equation}\label{eq:diagram-commutative}
\xymatrix{
\mathrm{SmProj}(k)\ar[dd]_-{M_R} \ar[rrr]^-{X \mapsto \perf_\dg(X)} &&& \dgcat_{\mathrm{sp}}(k) \ar[d]^-{U_R} \\
&&& \NMix(k;R)^\oplus \ar[d]^-{(-)^\vee} \\
\DM(k;R) \ar[rrr]_-{\Psi} &&& \NMix(k;R)^\oplus  \,,
}
\end{equation}
where $(-)^\vee$ stands for the duality (contravariant) functor. The functor $\Psi$ preserves moreover arbitrary direct sums.
\end{theorem}
\begin{remark}\label{rk:last}
As explained in Remark \ref{rk:dualizable} below, we have $U_R(\cA)^\vee \simeq U_R(\cA^\op)$ for every smooth proper dg category $\cA$.. Making use of the Morita equivalence between $\perf_\dg(X)^\op$ and $\perf_\dg(X)$, we hence conclude that $\Psi(M_R(X))\simeq U_R(\perf_\dg(X))$ for every smooth projective $k$-scheme $X$. 
\end{remark}

Theorem \ref{thm:main0} holds more generally with $\mathrm{SmProj}(k)$ replaced by smooth $k$-schemes; see \S\ref{sec:proof-thm0}. Consult also Theorem \ref{thm:orbit} for a factorization of $\Psi$ through an orbit category.
\begin{definition}\label{def:NMAM}
The category of {\em noncommutative mixed Artin motives $\NMAM(k;R)$} is defined as the smallest thick triangulated subcategory of $\NMix(k;R)$ containing the objects $U_R(l)$, where $l/k$ a finite separable field extension. In the same vein, let $\NMAM(k;R)^\oplus$ be the smallest triangulated subcategory of $\NMix(k;R)^\oplus$ which contains $\NMAM(k;R)$ and is closed under arbitrary direct~sums.
\end{definition}
As proved in Proposition \ref{prop:rel-Artin}, when $R$ is a $\bbQ$-algebra, the category $\NAM(k)_R$ can be identified with the smallest additive idempotent complete full subcategory of $\NMAM(k;R)$ containing the objects $U_R(l)$, where $l/k$ is a finite separable extension.

V.~Voevodsky remarked in \cite[page 217]{Voev} that, when $R$ is a $\bbQ$-algebra, the classical triangulated category of mixed Artin motives $\MAM(k;R) \subset \mathrm{DM}(k;R)$ is equivalent to the bounded derived category $\cD^b(\AM(k)_R)$. Since $\mathrm{AM}(k)_R$ is semi-simple, $\MAM(k;R)$ identifies then with the category $\mathrm{Gr}_\bbZ^b(\mathrm{AM}(k)_R)$ of bounded $\bbZ$-graded objects in $\mathrm{AM}(k)_R$. In particular, there are no higher Ext-groups. On the other hand, as explained in Example \ref{ex:schemes}, we have the following isomorphisms
\begin{equation}\label{eq:R-mod-isos}
\Hom_{\NMAM(k;R)}(U_R(\perf_\dg(X)), U_R(\perf_\dg(Y))[-n])\simeq K_n(X \times Y)_R
\end{equation}
for any two finite {\'e}tale $k$-schemes $X$ and $Y$. 
\begin{proposition}\label{prop:comparison}
When $k$ is finite, the functor $\Psi: \MAM(k;R) \to \NMAM(k;R)$ is an equivalence. Consequently, $\NMAM(k;R)$ identifies with $\mathrm{Gr}_\bbZ^b(\mathrm{AM}(k)_R)$. On the other hand, when $k$ is of characteristic zero, we~have~non-trivial~morphisms
$$ \Hom_{\NMAM(k;R)}(U_R(k),U_R(k)[-1]) \simeq K_1(k)_R \simeq k^\times \otimes_\bbZ R \neq 0\,.$$
Consequently, $\NMAM(k;R)$ is {\em not} equivalent to $\MAM(k;R)$. 
\end{proposition}
In contrast with the pure world, Proposition \ref{prop:comparison} shows that when $k$ is of characteristic zero the category of noncommutative mixed Artin motives is much richer than the classical category of mixed Artin motives. Roughly speaking, $\NMAM(k;R)$ contains not only $\MAM(k;R)$ but also all the higher algebraic $K$-theory groups of finite {\'e}tale $k$-schemes. For example, in the case of a number field $\bbF$, we have the following computation (see Borel \cite{Borel})
\begin{eqnarray*}
\Hom_{\NMAM(\bbQ;\bbQ)}(U_\bbQ(\bbQ),U_\bbQ(\bbF)[-n])\simeq \left\{ \begin{array}{ll}
         \bbQ^{r_2}& n\equiv 3\,\,\, (\mathrm{mod}\,4) \\
         \bbQ^{r_1+r_2} & n\equiv 1\,\,\, (\mathrm{mod}\,4) \\
         0 & \mathrm{otherwise}\end{array} \right. && n \geq 2\,,
\end{eqnarray*}
where $r_1$ (resp. $r_2$) stands for the number of real (resp. complex) embeddings of $\bbF$.
\subsection*{Motivic Hopf DG algebras}
The classical Tannakian formalism is quite restrictive since it requires the use of abelian categories. Consequently, it cannot be applied to the triangulated setting of noncommutative mixed motives. Fortunately, Ayoub \cite{Ayoub1,Ayoub2} has recently developed a weak Tannakian formalism for symmetric monoidal categories. Let $\omega:\cC \to \cD$ be a symmetric monoidal functor equipped with a right adjoint $\underline{\omega}$. Ayoub's weak Tannakian formalism asserts that under certain natural conditions on the functors $\omega$ and $\underline{\omega}$, the object $\cH(\cC):= (\omega\circ \underline{\omega})({\bf 1}) \in \cD$ becomes equipped with an Hopf algebra structure. Moreover, the functor $\omega$ admits a lifting  $\omega_{\mathrm{co}}: \cC \to \mathrm{coMod}(\cH(\cC)) $ to the category of $\cH(\cC)$-comodules and the Hopf algebra $\cH(\cC)$ is universal with respect to these properties; consult \cite[\S1]{Ayoub1} for details.
\begin{proposition}\label{prop:main}
When $R$ is a $k$-algebra, Hochschild homology $HH^k$ gives rise to a triangulated symmetric monoidal functor $HH^k_R: \NMix(k;R)^\oplus \to \cD(R)$ which satisfies the conditions of Ayoub's weak Tannakian formalism. Consequently, we obtain a motivic Hopf dg algebra $\cH(\NMix(k;R)^\oplus)\in \cD(R)$.
\end{proposition}
The motivic Hopf dg algebra $\cH(\NMix(k;R)^\oplus)$ is the mixed analogue of the motivic Galois group $\mathrm{Gal}(\NNum^\dagger(k)_R)$. Note that while the existence of the latter affine group scheme is conditional to the noncommutative standard conjectures $C_{NC}$ and $D_{NC}$, the former Hopf dg algebra is always defined! Moreover, in contrast with the mysterious structure of the motivic Galois group, the motivic Hopf dg algebra admits the following explicit description: given dg categories $\cA$ and $\cB$, recall from \S\ref{sub:cubes} the construction of the complex of exact cubes $\square_R(\cA,\cB)$.
\begin{proposition}\label{prop:explicit}
The sum-total complex of the following simplicial object 
\begin{eqnarray*}
n & \mapsto & \bigoplus_{\cA_0, \ldots, \cA_n} HH^k_R(\cA_0^\op) \otimes \square_R(\cA_0,A_1) \otimes \cdots \otimes \square_R(\cA_{n-1}, \cA_n) \otimes HH^k_R(\cA_n)\,,
\end{eqnarray*}
where $\cA_0, \ldots, \cA_n \in \dgcat_{\mathrm{sp}}(k)$, is naturally a dg bialgebra $\mathrm{C}$. The multiplication is induced by the tensor product of dg categories and the comultiplication by insertion of the identity elements $HH^k_R(Id_{\cA_i}) \in HH^k_R(\cA_i) \otimes HH^k_R(\cA_i^\op)$. In the derived category $\cD(R)$, $\mathrm{C}$ becomes isomorphic to the dg bialgebra $\cH(\NMix(k;R)^\oplus)$.
\end{proposition}
Note that Proposition \ref{prop:explicit} renders the motivic dg bialgebra strict. Its proof is based on Pridham's work \cite{Pridham}. As explained in \S\ref{sub:cubes}, the complex of exact cubes computes the spectrum homology of algebraic $K$-theory. Therefore, Proposition \ref{prop:explicit} also shows that the motivic dg bialgebra can be explicitly described using solely Hochschild homology and algebraic $K$-theory of smooth proper~dg~categories.

Now, let $\mathrm{MATM}(k;R)$ be the classical triangulated category of mixed Artin-Tate motives (see Wildeshaus \cite{Wild}) and $\mathrm{MATM}(k;R)^\oplus$ the smallest triangulated subcategory of $\mathrm{DM}(k;R)$ which contains $\mathrm{MATM}(k;R)$ and is closed under arbitrary direct sums. Under these notations, the motivic Hopf dg algebra associated to the category of noncommutative mixed Artin motives admits the following description:


\begin{theorem}\label{thm:main1}
When $R$ is a $k$-algebra, the following functors
\begin{eqnarray*}\label{eq:composed}
\NMAM(k;R)^\oplus \stackrel{HH^k_R}{\too} \cD(R) && \mathrm{MATM}(k;R)^\oplus \stackrel{\Psi}{\too} \mathrm{NMAM}(k;R)^\oplus \stackrel{HH^k_R}{\too} \cD(R)
\end{eqnarray*}
satisfy the conditions of Ayoub's weak Tannakian formalism. Moreover, we have an isomorphism of Hopf dg algebras 
$$\cH(\NMAM(k;R)^\oplus) \simeq \cH(\mathrm{MATM}(k;R)^\oplus)\otimes_{R[t,t^{-1}]}R\,,$$ where the $R[t,t^{-1}]$-linear structure is induced by the $\otimes$-invertible Tate motive $R(1)[2]$.
\end{theorem}
\begin{remark}
Let $X \in \mathrm{SmProj}(k)$ with $k$ of characteristic zero. As proved by Weibel in \cite[Prop.~1.3 and Cor.~1.4]{Weibel1}, the $n^{\mathrm{th}}$ homology of $HH^k(\perf_\dg(X))\simeq HH^k(X)$ identifies with $\bigoplus_{p-q=n} H^q(X,\Omega^p_X)$. Making use of Remark \ref{rk:last}, we hence conclude that the above right-hand side functor maps $M_R(X)$ to $\bigoplus_n \bigoplus_{p-q=n} H^q(X,\Omega_X^p)_R$.
\end{remark}
Roughly speaking, Theorem \ref{thm:main1} shows that the default between noncommutative mixed Artin motives and mixed Artin-Tate motives is measured solely by the existence of an $R[t,t^{-1}]$-linear structure. Its proof is based on the behavior of Ayoub's weak Tannakian formalism with respect to orbit categories; see~Proposition~\ref{prop:orbit}.




Let $\cC^0(\mathrm{Gal}(\overline{k}/k),R)$ be the Hopf $R$-algebra of continuous functions with multiplication (resp. comultiplication) induced by the multiplication in $R$ (resp. in $\mathrm{Gal}(\overline{k}/k)$), and $\mathrm{MAM}(k;R)^\oplus$ the smallest triangulated subcategory of $\mathrm{DM}(k;R)$ which contains $\mathrm{MAM}(k;R)$ and is closed under arbitrary direct sums, and . 
\begin{theorem}\label{thm:new}
When $R$ is a $k$-algebra, the following functor
\begin{equation}\label{eq:functor-Mix-Artin}
\mathrm{MAM}(k;R)^\oplus \stackrel{\Psi}{\too} \mathrm{NMAM}(k;R)^\oplus \stackrel{HH^k_R}{\too} \cD(R)
\end{equation}
satisfies the conditions of Ayoub's weak Tannakian formalism. Moreover, when $k \subset \bbC$ and $R$ is a $\bbC$-algebra, the associated Hopf dg algebra $\cH(\mathrm{MAM}(k;R)^\oplus)$ is concentrated in degree zero and agrees with $\cC^0(\mathrm{Gal}(\overline{k}/k),R)$.
\end{theorem}
The inclusion of categories $\mathrm{MAM}(k;R)^\oplus \subset \mathrm{MATM}(k;R)^\oplus$ gives rise to an injective morphism of Hopf dg algebras $\cC^0(\mathrm{Gal}(\overline{k}/k),R) \to \cH(\mathrm{MATM}(k;R)^\oplus)$. Hence, the above Theorems \ref{thm:main1} and \ref{thm:new} show that, in contrast with the pure world, the motivic Hopf dg algebra  $\cH(\NMAM(k;R)^\oplus)$ carries much more information than just continuous functions on the absolute Galois group. Roughly speaking, all the above extensions \eqref{eq:R-mod-isos} contribute to the motivic Hopf dg algebra. 

Finally, the short exact sequence \eqref{eq:short-exact} admits the following mixed analogue:
\begin{theorem}\label{thm:main2}
When $k\subset \bbC$ and $R$ is a $\bbC$-algebra, we have the following short exact sequence of Hopf dg algebras
\begin{equation*}\label{eq:short-exact2}
\quad \quad 1 \too \cC^0(\mathrm{Gal}(\overline{k}/k),R) \too \cH(\NMix(k;R)^\oplus) \too \cH(\NMix(\overline{k};R)^\oplus) \too 1\,,
\end{equation*}
where the first map is induced by the functor $\Psi: \mathrm{MAM}(k;R)^\oplus \to \NMix(k;R)^\oplus$ and the second one by the base-change functor $-\otimes_k \overline{k}: \NMix(k;R)^\oplus \to \NMix(\overline{k};R)^\oplus$.
\end{theorem}
Intuitively speaking, Theorem \ref{thm:main2} shows that the default between noncommutative mixed motives over $k$ and noncommutative mixed motives over $\overline{k}$ is measured solely by the profinite absolute Galois group.
\medbreak\noindent\textbf{Acknowledgments:} The author is very grateful to Joseph Ayoub for kindly teaching him his beautiful weak Tannakian formalism \cite{Ayoub1,Ayoub2} and also for comments and corrections on a previous version of this article. He would also like to thank Clark Barwick, Guillermo Corti{\~n}as, Haynes Miller, Amnon Neeman, and Niranjan Ramachandran for useful conversations.
\section{Preliminaries}
\subsection{Notations}
Throughout the article, $k$ will be a base (perfect) field and $R$ a commutative ring of coefficients. Given a closed symmetric monoidal category $(\cC,\otimes, {\bf 1})$, we will write $\uHom(-,-)$ for its internal Hom and $(-)^\vee:=\uHom(-,{\bf 1})$ for the duality functor. When $\cC$ is enriched over a symmetric monoidal category $\cE$, we will write $\uHom_\cE(-,-)$ for this enrichment. We will use freely the language of Quillen model categories  \cite{Goerss,Hirschhorn,Hovey-book,Quillen}, of exact categories \cite{Quillen1}, of Waldhausen categories \cite{Wald},  and of $\infty$-categories \cite{Lurie1,Lurie2}. Given a Quillen model category $\cC$, we will write $\Ho(\cC)$ for its homotopy category and $\cC^\infty$ for the associated $\infty$-category. Finally, adjunctions will be displayed vertically with the left (resp. right) adjoint on the left (resp. right) hand side. 
\subsection{Dg categories}\label{sec:dg}
Let $\cC(k)$ be the category of cochain complexes of $k$-vector spaces. A {\em differential graded (=dg) category $\cA$} is a $\cC(k)$-enriched category and a {\em dg functor} $F:\cA\to \cB$ is a $\cC(k)$-enriched functor; consult Keller's ICM survey \cite{ICM-Keller}.

Let $\cA$ be a dg category. The category $\dgHo^0(\cA)$ has the same objects as $\cA$ and $\dgHo^0(\cA)(x,y):=H^0\cA(x,y)$. The opposite dg category $\cA^\op$ has the same objects as $\cA$ and $\cA^\op(x,y):=\cA(y,x)$. A  right $\cA$-module is a dg functor $\cA^\op \to \cC_\dg(k)$ with values in the dg category $\cC_\dg(k)$ of cochain complexes of $k$-vector spaces. Let $\cC(\cA)$ be the category of right $\cA$-modules. As explained in \cite[\S3.1-3.2]{ICM-Keller}, $\cC(\cA)$ carries a projective Quillen model structure. Moreover, the dg structure of $\cC_\dg(k)$ makes $\cC(\cA)$ into a dg category $\cC_\dg(\cA)$. The derived category $\cD(\cA)$ of $\cA$ is defined as the localization of $\cC(\cA)$ with respect to the objectwise quasi-isomorphisms. Its full triangulated subcategory of compact objects will be denoted by $\cD_c(\cA)$.

A dg functor $F:\cA\to \cB$ is called a {\em Morita equivalence} if it induces an equivalence on derived categories $\cD(\cA) \stackrel{\simeq}{\to} \cD(\cB)$; see \cite[\S4.6]{ICM-Keller}. As proved in \cite[Thm.~5.3]{IMRN}, $\dgcat(k)$ admits a Quillen model structure whose weak equivalences are the Morita equivalences. Let us write $\Hmo(k)$ for the associated homotopy category.

The tensor product $\cA\otimes\cB$ of dg categories is defined as follows: the set of objects is the cartesian product of the sets of objects and $(\cA\otimes\cB)((x,w),(y,z)):= \cA(x,y) \otimes \cB(w,z)$. As explained in \cite[\S2.3 and \S4.3]{ICM-Keller}, this construction gives rise to symmetric monoidal structure on $\dgcat(k)$ which descends to $\Hmo(k)$.

An $\cA\text{-}\cB$-bimodule is a dg functor $\mathrm{B}:\cA \otimes \cB^\op\to \cC_\dg(k)$ or equivalently a right $(\cA^\op \otimes \cB)$-module. A standard example is the $\cA\text{-}\cB$-bimodule
\begin{eqnarray}\label{eq:bimodules111}
{}_F\mathrm{B}:\cA\otimes \cB^\op \too \cC_\dg(k) && (x,w) \mapsto \cB(w,F(x))
\end{eqnarray}
associated to a dg functor $F:\cA \to \cB$. Let us denote by $\rep(\cA,\cB)$ the full triangulated subcategory of $\cD(\cA^\op \otimes \cB)$ consisting of those $\cA\text{-}\cB$-bimodules $\mathrm{B}$ such that $\mathrm{B}(x,-) \in \cD_c(\cB)$ for every object $x \in \cA$. In the same vein, let $\rep_\dg(\cA,\cB)$ be the full dg subcategory of $\cC_\dg(\cA^\op \otimes \cB)$ consisting of those $\cA\text{-}\cB$-bimodules which belong to $\rep(\cA,\cB)$. By construction, we have $\dgHo^0(\rep_\dg(\cA,\cB))\simeq \rep(\cA,\cB)$.
\subsection{Finite dg cells}\label{sub:cells}
For $n \in \mathbb{Z}$, let $S^{n}$ be the cochain complex $k[n]$ and $D^n$ the mapping
cone of the identity on $S^{n-1}$. Let $\dgS(n)$ be the dg
category with objects $1,2$ such that $ \dgS(n)(1,1)=k \ko
\dgS(n)(2,2)=k \ko \dgS(n)(2,1)=0  \ko \dgS(n)(1,2)=S^{n} $ and with composition given by multiplication. Similarly, let $\dgD(n)$ be the dg
category with objects $3,4$ such that $ \dgD(n)(3,3)=k \ko
\dgD(n)(4,4)=k \ko \dgD(n)(4,3)=0$ and $\dgD(n)(3,4)=D^n $. For $n \in \bbZ$, let $\iota(n):\dgS(n-1)\to \dgD(n)$ be the dg functor that sends $1$ to
$3$, $2$ to $4$ and $S^{n-1}$ to $D^n$ by the identity on $k$ in
degree~$n-1$\,:
$$
\vcenter{
\xymatrix@C=2em@R=1em{
\dgS(n-1) \ar@{=}[d] \ar[rrr]^{\displaystyle \iota(n)}
&&& \dgD(n) \ar@{=}[d]
\\
&&&\\
\\
1 \ar@(ul,ur)[]^{k} \ar[dd]_-{S^{n-1}}
& \ar@{|->}[r] &
& 3\ar@(ul,ur)[]^{k}  \ar[dd]^-{D^n}
\\
& \ar[r]^-{\incl} &
\\
2 \ar@(dr,dl)[]^{k}
& \ar@{|->}[r] &
& 4\ar@(dr,dl)[]^{k}
}}
\qquad\text{where}\qquad
\vcenter{\xymatrix@R=1em@C=.8em{ S^{n-1} \ar[rr]^-{\incl} \ar@{=}[d]
&& D^n \ar@{=}[d]
\\
\ar@{.}[d]
&& \ar@{.}[d]
\\
0 \ar[rr] \ar[d]
&& 0 \ar[d]
\\
0 \ar[rr] \ar[d]
&& k \ar[d]^{\id}
\\
k \ar[rr]^{\id} \ar[d]
&& k \ar[d]
&{\scriptstyle(\textrm{degree }n-1)}
\\
0 \ar[rr] \ar@{.}[d]
&& 0 \ar@{.}[d]
\\
&&}}
$$
A dg category $\cA$ is called a \emph{finite dg cell} if the unique dg functor $\emptyset\to\cA$ (where the empty dg category $\emptyset$ is the initial object of $\dgcat(k)$) can be expressed as a finite composition of pushouts along the dg functors $\iota(n), n \in \bbZ$, and $\emptyset \to k$.
\subsection{Orbit categories}\label{sub:orbit}
Let $\cC$ be an additive symmetric monoidal category and $O \in \cC$ a $\otimes$-invertible object. Recall from \cite{CvsNC} that the {\em orbit category} $\cC\!/_{\!\!-\otimes O}$ has the same objects as $\cC$ and morphisms given by
$$\Hom_{\cC\!/_{\!\!-\otimes O}}(a,b):= \bigoplus_{i \in \bbZ} \Hom_\cC(a,b\otimes O^{\otimes i})\,.$$
The composition is induced from $\cC$. 
By construction, $\cC\!/_{\!\!-\otimes O}$ is additive, symmetric monoidal (see \cite[Lem.~7.3]{CvsNC}), and comes equipped with a canonical projection symmetric monoidal functor $\pi: \cC \to \cC\!/_{\!\!-\otimes O}$. Moreover, $\pi$ is endowed with a natural $2$-isomorphism $\pi\circ (-\otimes O) \stackrel{\sim}{\Rightarrow} \pi$ and is $2$-universal among all such functors.
\section{Noncommutative mixed motives with coefficients}\label{sec:NCMixed}
Recall from \cite[Def.~15.1]{Duke} \cite[Thm.~8.5]{CT} the construction of the closed  symmetric monoidal Quillen model category $\mathrm{Mot}(k):= L_\add \Fun (\dgcat_{\mathrm{f}}(k)^\op, \Spt)$ and of the symmetric monoidal functor 
\begin{eqnarray*}
U: \dgcat(k) \too \mathrm{Mot}(k) && \cB \mapsto (\cA \mapsto \Sigma^\infty(N(w\rep_\dg(\cA,\cB))_+))\,.
\end{eqnarray*}
Several explanations are in order: $\dgcat_{\mathrm{f}}(k)$ is obtained by stabilizing the finite dg cells with respect to tensor products, (co)fibrant resolutions, and left framings; $\Spt$ is the projective stable Quillen model category of symmetric spectra \cite{HSS}; $w\rep_\dg(\cA,\cB)$ is the subcategory of quasi-isomorphisms of $\rep_\dg(\cA,\cB)$; $N(-)$ is the nerve functor; $\Sigma^\infty(-_+)$ is the suspension symmetric spectrum functor; and $L_{\add}$ stands for the left Bousfield localization of $\Fun (\dgcat_{\mathrm{f}}(k)^\op, \Spt)$ with respect to the stabilization under (de)suspensions of the following sets of morphisms (consult \cite[\S13]{Duke} for the notion of a short exact sequence of dg categories)
$$
\{U(F):U(\cA) \to U(\cB)\,|\,F\,\,\mathrm{is}\,\,\mathrm{a}\,\,\mathrm{Morita}\,\,\mathrm{equivalence}\,\,\,\,\cA, \cB \in \dgcat_{\mathrm{f}}(k)\}
$$
$$ \{ \mathrm{cofiber} (U\left(\cA) \to U(\cB)\right) \to U(\cC)\,|\, \xymatrix@C=1.5em{
0 \ar[r] & \cA \ar[r] & \cB \ar@/_0.5pc/[l]
\ar[r] & \cC \ar@/_0.5pc/[l] \ar[r]& 0
}\,\, \cA, \cB, \cC \in \dgcat_{\mathrm{f}}(k) \}$$
Finally, the symmetric monoidal structure is given by the Day convolution product. The associated homotopy category $\Ho(\Mot(k))$ is naturally enriched over $\Ho(\Spt)$. Given dg categories $\cA$ and $\cB$, with $\cA$ smooth proper, we have an isomorphism
\begin{equation}\label{eq:spectra}
\uHom_{\Ho(\Spt)}(U(\cA),U(\cB))\simeq K(\cA^\op \otimes \cB)\,,
\end{equation}
Moreover, the composition law is induced by the tensor product of bimodules.

Following Kontsevich \cite{Miami,IAS,finMot}, the category of noncommutative mixed motives $\NMix(k)$ was introduced in \cite[\S9.2]{CT} as the smallest thick triangulated subcategory of $\Ho(\mathrm{Mot}(k))$ containing the objects $\{U(\cA)\,|\, \cA\in \dgcat_{\mathrm{sp}}(k)\}$. 
\subsection{Coefficients}\label{sub:coefficients}
Let $HR \in \Spt$ be the Eilenberg-MacLane ring spectrum of $R$ and $\mathrm{Mod}(HR)$ the symmetric monoidal Quillen model category of $HR$-modules; see \cite[Cor.~5.4.2]{HSS}. By construction, we have the following Quillen adjunction:
\begin{equation}\label{eq:Quillen-1}
\xymatrix{
\mathrm{Mod}(HR) \ar@<1ex>[d]^-{\mathrm{forget}} \\
\Spt \ar@<1ex>[u]^-{-\wedge HR}\,.
}
\end{equation}
\begin{definition}\label{def:model-R}
Let $\Mot(k;R)$ be the closed symmetric monoidal Quillen model category $L_\add\Fun(\dgcat_{\mathrm{f}}(k)^\op, \mathrm{Mod}(HR))$ and 
\begin{eqnarray*}
U_R: \dgcat(k) \too \mathrm{Mot}(k;R) && \cB \mapsto (\cA \mapsto \Sigma^\infty(N(w\rep_\dg(\cA,\cB))_+)\wedge HR)
\end{eqnarray*}
the associated symmetric monoidal functor.
\end{definition}

\begin{remark}[Alternative Quillen model]
The Quillen model category $\mathrm{Mod}(HR)$ is Quillen equivalent to the projective Quillen model category $\cC(R)$ of cochain complexes of $R$-modules; see \cite[Thm.~5.1.6]{SS}. Therefore, if in Definition \ref{def:model-R} we replace $\mathrm{Mod}(HR)$ by $\cC(R)$, we obtain a Quillen equivalent model category.
\end{remark}
Note that $\Ho(\Mot(k;R))$ is naturally enriched over $\Ho(\mathrm{Mod}(HR))\simeq \cD(R)$. The above adjunction \eqref{eq:Quillen-1} gives rise to the following Quillen adjunction
\begin{equation}\label{eq:Quillen-2}
\xymatrix{
\Mot(k;R) \ar@<1ex>[d]^-{\mathrm{forget}} \\
\Mot(k) \ar@<1ex>[u]^-{-\wedge HR}\,.
}
\end{equation}
\begin{remark}[Strongly dualizable objects]\label{rk:dualizable}
As proved in \cite[Thm.~5.8]{CT}, the smooth proper dg categories can be characterized as the strongly dualizable objects of the symmetric monoidal category $\Hmo(k)$. The dual of a smooth proper dg category $\cA$ is given by the opposite dg category $\cA^\op$. Since $U_R$ is symmetric monoidal, we hence conclude that the objects $\{U_R(\cA)\,|\, \cA\in \dgcat_{\mathrm{sp}}(k)\}$ are strongly dualizable in the homotopy category $\Ho(\Mot(k;R))$. 
\end{remark}
\begin{proposition}\label{prop:computation}
Given dg categories $\cA$ and $\cB$, with $\cA$ smooth proper, we have
\begin{equation}\label{eq:enrichement-R}
\uHom_{\cD(R)}(U_R(\cA),U_R(\cB)) \simeq K(\cA^\op \otimes \cB) \wedge HR\,.
\end{equation}
When $R$ is a $\bbQ$-algebra, we have moreover $R$-module isomorphisms
\begin{eqnarray*}
\Hom_{\Ho(\Mot(k;R))}(U_R(\cA), U_R(\cB)[-n]) \simeq K_n(\cA^\op \otimes \cB)_R && n \geq 0\,.
\end{eqnarray*}
\end{proposition}
\begin{remark}
Similarly to $\Ho(\Mot(k))$, the composition law of the homotopy category $\Ho(\Mot(k;R))$ is induced by the tensor product of bimodules.
\end{remark}
\begin{proof}
As mentioned in Remark \ref{rk:dualizable}, the object $U_R(\cA)$ is strongly dualizable with dual $U_R(\cA^\op)$. Hence, we have the following isomorphisms
\begin{eqnarray}
\uHom_{\cD(R)}(U_R(\cA),U_R(\cB)) & \simeq & \uHom_{\cD(R)}(U_R(k), U_R(\cA^\op)\otimes U_R(\cB)) \nonumber \\
& \simeq & \uHom_{\cD(R)}(U_R(k), U_R(\cA^\op \otimes \cB)) \nonumber \\
& \simeq & \uHom_{\Ho(\Spt)}(U(k),U_R(\cA^\op \otimes \cB)) \label{eq:star-2} \\
& \simeq & K(\cA^\op \otimes \cB) \wedge HR\,,\label{eq:star-3}
\end{eqnarray}
where \eqref{eq:star-2} follows from the above adjunction \eqref{eq:Quillen-2} and \eqref{eq:star-3} from isomorphism \eqref{eq:spectra}. This proves the first claim. In what concerns the second claim, note that \eqref{eq:enrichement-R} gives rise to $R$-module isomorphisms
\begin{eqnarray*}
\Hom_{\Ho(\Mot(k;R))}(U_R(\cA), U_R(\cB)[-n]) \simeq \pi_n(K(\cA^\op \otimes \cB)\wedge HR)  && n\geq 0\,.
\end{eqnarray*}
When $R=\bbQ$, the right-hand side identifies with $K_n(\cA^\op \otimes \cB) \otimes_\bbZ \bbQ$; see \cite[page~203]{Adams}. Therefore, in the particular case where $R$ is a $\bbQ$-algebra, we conclude that 
$$ \pi_n(K(\cA^\op \otimes \cB)\wedge HR)\simeq \pi_n(K(\cA^\op \otimes \cB)\wedge H\bbQ)\otimes_\bbQ R  \simeq K_n(\cA^\op \otimes \cB)_R\,.$$
This proves the second claim.
\end{proof}
\begin{example}[Schemes]\label{ex:schemes}
Let $X,Y$ be smooth projective $k$-schemes and $R$ a $\bbQ$-algebra. By combining Proposition \ref{prop:computation} with the Morita equivalences $\perf_\dg(X)^\op \simeq \perf_\dg(X)$ and $\perf_\dg(X)\otimes \perf_\dg(Y) \simeq \perf_\dg(X \times Y)$, we hence conclude that 
\begin{eqnarray*}
\Hom_{\Ho(\Mot(k;R))}(U_R(\perf_\dg(X)),U_R(\perf_\dg(Y))[-n])\simeq K_n(X \times Y)_R && n \geq 0\,.
\end{eqnarray*}
\end{example}
\begin{definition}\label{def:NCmixed}
The category of {\em noncommutative mixed motives $\NMix(k;R)$} is the smallest thick triangulated subcategory of $\Ho(\mathrm{Mot}(k;R))$ containing the objects $\{U_R(\cA)\,|\, \cA\in \dgcat_{\mathrm{sp}}(k)\}$. In the same vein, let $\NMix(k;R)^\oplus$ be the smallest triangulated subcategory of $\Ho(\Mot(k;R))$ which contains $\NMix(k;R)$ and is closed under arbitrary direct sums. Since the smooth proper dg categories are stable under tensor product, the triangulated categories $\NMix(k;R)$ and $\NMix(k;R)^\oplus$ are moreover symmetric monoidal.
\end{definition}
\begin{remark}[Compact generators]\label{rk:symmetric}
As proved in \cite[Cor.~8.7]{CT}, the objects $U(\cA)$, with $\cA \in\dgcat_{\mathrm{sp}}(k)$, are compact. Since the left adjoint functor in \eqref{eq:Quillen-2} preserves compact objects, we hence conclude that $\{U_R(\cA)\,|\, \cA \in \dgcat_{\mathrm{sp}}(k) \}$ is a set of compact generators of the triangulated category $\NMix(k;R)^\oplus$.
\end{remark}
\subsection{Complex of exact cubes}\label{sub:cubes}
Let $\cE$ be an exact category and $\langle-1,0,1\rangle$ the ordered set of three elements (considered as a category). Following Takeda \cite[\S2.1]{Takeda}, a {\em $n$-cube of $\cE$} consists of a covariant functor $\cF:\langle-1,0,1\rangle^n \to \cE$. To every element $(\alpha_1, \ldots, \alpha_n) \in \langle-1,0,1\rangle^{n-1}$ and $1\leq i \leq n$, we have an associated edge
$$ 
\cF_{(\alpha_1, \ldots, \alpha_{i-1}, -1, \alpha_i, \ldots, \alpha_{n-1})} \too 
 \cF_{(\alpha_1, \ldots, \alpha_{i-1}, 0, \alpha_i, \ldots, \alpha_{n-1})} \too
 \cF_{(\alpha_1, \ldots, \alpha_{i-1}, 1, \alpha_i, \ldots, \alpha_{n-1})}\,.
$$
An $n$-cube is called {\em exact} if all its edges are short exact sequences in $\cE$. Let $C_n(\cE)$ be the set of exact $n$-cubes of $\cE$. As explained in {\em loc. cit.}, $\{C_n(\cE)\}_{n \geq 0}$ is naturally a simplicial set. Let us denote by $\square_R(\cE)$ the associated normalized complex of $R$-modules. As proved by McCarthy in \cite[\S3]{McCarthy}, this complex computes the spectrum homology of the algebraic $K$-theory of $\cE$, \ie $H_\ast(\square_R(\cE))\simeq \pi_\ast(K(\cE)\wedge HR)$. 

Let $(\cE,w)$ be a Waldhausen category. Following Gillet-Soul{\'e} \cite[\S6.2]{G-S}, let  $\square_R(\cE,w):=\square_R(\cE)/\square_R(\cE^w)$. As explained in {\em loc. cit.}, we have also $H_\ast(\square_R(\cE,w))\simeq \pi_\ast(K(\cE,w)\wedge HR)$. Moreover, every biexact functor $(\cE,w)\times (\cE',w') \to (\cE'',w'')$ gives rise to a bilinear pairing $\square_R(\cE,w) \otimes_R \square_R(\cE',w')\to \square_R(\cE'',w'')$.

Now, let $\cA$ be a dg category. As explained by Dugger-Shipley in \cite[\S3]{DS}, the full subcategory of $\cC(\cA)$ consisting of the cofibrant right $\cA$-modules which belong to $\cD_c(\cA)$ is naturally a Waldhausen category. Let us then write $\square_R(\cA)$ for the associated complex of exact cubes. Similarly, given dg categories $\cA$ and $\cB$, let $\square_R(\cA,\cB):= \square_R(\cA^\op \otimes \cB)$. As mentioned above, we have $H_\ast(\square_R(\cA,\cB))\simeq \pi_\ast(K(\cA^\op \otimes \cB)\wedge HR)$. Therefore, the complex of $R$-modules $\square_R(\cA,\cB)$ should be considered as an explicit model of the left-hand side of \eqref{eq:enrichement-R}. Under this model, the composition law corresponds to the bilinear pairings
\begin{eqnarray*}
\square_R(\cA,\cB) \otimes_R \square_R(\cB,\cC) \too \square_R(\cA,\cC) && \cA, \cB, \cC \in \dgcat_{\mathrm{sp}}(k)
\end{eqnarray*}
induced by the tensor product of bimodules. We obtain in this way an explicit differential graded enhancement\footnote{We are implicitly strictifying the tensor product of bimodules in order to make it strictly associative; see Garkusha-Panin \cite[\S3]{GP} for instance.} of the full subcategory of $\NMix(k;R)$ consisting of the smooth proper dg categories.
\subsection{Universal property}\label{sub:universal}
Let $\dgcat(k)^\infty$ be the symmetric monoidal $\infty$-category of dg categories. A functor $E:\dgcat(k)^\infty \to \mathrm{D}$, with values in a presentable stable $\infty$-category, is called an {\em additive invariant} if it inverts the Morita equivalences, preserves filtered colimits, and sends split short exact sequences of dg categories to cofiber sequences. When $E$ is moreover symmetric monoidal, we call it a {\em symmetric monoidal additive invariant}. As proved in \cite[Thm.~15.4]{Duke}\cite[Thm.~8.5]{CT}, the functor $U:\dgcat(k)^\infty \to \Mot(k)^\infty$ is the {\em universal additive invariant}, \ie given any (symmetric monoidal) presentable stable $\infty$-category $\mathrm{D}$, we have equivalences
\begin{eqnarray*}
U^\ast: \Fun^L(\Mot(k)^\infty, \mathrm{D})  & \stackrel{\simeq}{\too} &  \Fun^{\add}(\dgcat(k)^\infty, \mathrm{D})\\
U^\ast: \Fun^{L,\otimes}(\Mot(k)^\infty, \mathrm{D}) & \stackrel{\simeq}{\too} & \Fun^{\add, \otimes}(\dgcat(k)^\infty, \mathrm{D})\,,
\end{eqnarray*}
where the left-hand side denotes the $\infty$-category of colimit preserving (symmetric monoidal) functors and the right-hand side the $\infty$-category of (symmetric monoidal) additive invariants. Note that in the same way the presentable $\infty$-category $\Mot(k)^\infty$ is {\em stable}, \ie it is a module over $\Spt^\infty$ in the sense of \cite[\S4]{Lurie2}, the presentable $\infty$-category $\Mot(k;R)^\infty$ is a {\em $R$-module}, \ie it is a module over $\mathrm{Mod}(HR)^\infty$. Hence, the universal property of the functor $U_R: \dgcat(k)^\infty \to \Mot(k;R)^\infty$ is the following: given any (symmetric monoidal) presentable $R$-module $\infty$-category $\mathrm{D}$, we have
\begin{eqnarray}
U_R^\ast: \Fun^L_R(\Mot(k;R)^\infty, \mathrm{D}) & \stackrel{\simeq}{\too} & \Fun^{\add}(\dgcat(k)^\infty, \mathrm{D})\label{eq:equivalence-univ0} \\
U_R^\ast: \Fun^{L,\otimes}_R(\Mot(k;R)^\infty, \mathrm{D}) & \stackrel{\simeq}{\too} & \Fun^{\add,\otimes}(\dgcat(k)^\infty, \mathrm{D})\,, \label{eq:equivalence-univ}
\end{eqnarray}
where the left-hand side denotes the $\infty$-category of colimit preserving (symmetric monoidal) $R$-module functors and the right-hand side the $\infty$-category of (symmetric monoidal) additive invariants.
\subsection{Base-change}\label{sub:base-change}
Given a field extension $l/k$, consider the following colimit preserving symmetric monoidal base-change functor 
\begin{eqnarray}\label{eq:base-change}
-\otimes_k l: \dgcat(k) \too \dgcat(l) && \cA \mapsto \cA \otimes_k l\,.
\end{eqnarray}
As proved in \cite[Prop.~7.1]{Artin}, the above functor \eqref{eq:base-change} preserves Morita equivalences. Thanks to the work of Drinfeld \cite[Prop.~1.6.3]{Drinfeld}, it also preserves (split) short exact sequences of dg categories. As mentioned in Remark \ref{rk:dualizable}, the smooth proper dg categories can be characterized as the strongly dualizable object of $\Hmo(k)$ (and $\Hmo(l)$). Therefore, since the above functor \eqref{eq:base-change} is symmetric monoidal and preserves Morita equivalences, we conclude that it preserves also the smooth proper dg categories. Making use of the above equivalence of categories \eqref{eq:equivalence-univ}, we hence obtain an $R$-linear symmetric monoidal triangulated functor (which we still denote by $-\otimes_k l$) making the following diagram commute:
\begin{equation}\label{eq:diagram-base-change}
\xymatrix{
\dgcat_{\mathrm{sp}}(k) \ar[d]_-{U_R} \ar[rr]^-{-\otimes_k l} && \dgcat_{\mathrm{sp}}(l) \ar[d]^-{U_R} \\
\NMix(k;R)^\oplus \ar[rr]_-{-\otimes_k l} && \NMix(l;R)^\oplus\,.
}
\end{equation}
\subsection{Relation with noncommutative Chow motives}\label{sub:relation}
There is a bijection $\Hom_{\Hmo(k)}(\cA,\cB)\simeq \mathrm{Iso}\,\rep(\cA,\cB)$ under which the composition law of $\Hmo(k)$ corresponds to the tensor product of bimodules; see \cite[Cor.~5.10]{IMRN}. Since the $\cA\text{-}\cB$-bimodules \eqref{eq:bimodules111} belong to
$\rep(\cA,\cB)$, we obtain the symmetric monoidal functor
\begin{eqnarray}\label{eq:functor1}
\dgcat(k) \too \Hmo(k) & \cA \mapsto \cA & F \mapsto {}_F\mathrm{B}\,.
\end{eqnarray}
Let $\Hmo_0(k)_R$ be the $R$-linear additive category with the same objects as $\Hmo(k)$ and morphisms $\Hom_{\Hmo_0(k)_R}(\cA,\cB):=K_0(\rep(\cA,\cB))_R$, where $K_0\rep(\cA,\cB)$ stands for the Grothendieck group of the triangulated category $\rep(\cA,\cB)$. The composition law is induced by the tensor product of bimodules and the symmetric monoidal structure extends by $R$-bilinearity from $\Hmo(k)$ to $\Hom_0(k)_R$. We have also the following symmetric monoidal functor
\begin{eqnarray}\label{eq:nat2}
\Hmo(k) \too \Hmo_0(k)_R &\cA \mapsto \cA& \mathrm{B} \mapsto [\mathrm{B}]_R\,.
\end{eqnarray}
Let us denote by $\mathrm{U}_R: \dgcat(k) \to \Hmo_0(k)_R$ the composition of $\eqref{eq:nat2} \circ \eqref{eq:functor1}$. Recall from the surveys \cite{survey2, survey} that the category of {\em noncommutative Chow motives $\mathrm{NChow}(k)_R$} is defined as the idempotent completion of the full subcategory of $\Hmo_0(k)_R$ consisting of the objects $\{\mathrm{U}_R(\cA)\,|\, \cA \in \dgcat_{\mathrm{sp}}(k)\}$.
\begin{proposition}\label{prop:base-change}
We have an $R$-linear additive symmetric monoidal functor
\begin{eqnarray}\label{eq:pure-mixed}
\mathrm{NChow}(k)_R \too \NMix(k;R)^\oplus && \mathrm{U}_R(\cA) \mapsto U_R(\cA)\,.
\end{eqnarray}
\end{proposition}
When $R$ is a $\bbQ$-algebra, the above functor \eqref{eq:pure-mixed} is moreover fully-faithful. 
\begin{proof}
Consider the following composition 
\begin{equation*}\label{eq:composition-new}
\dgcat(k) \stackrel{U_R}{\too} \Mot(k;R) \too \Ho(\Mot(k;R))\,.
\end{equation*}
Since the category $\Ho(\Mot(k;R))$ is additive and the above composition inverts Morita equivalences, sends split short exact sequences of dg categories to direct sums, and is moreover symmetric monoidal, \cite[Thms.~5.3 and 6.3]{IMRN} furnishes us a (unique) $R$-linear additive symmetric monoidal functor $\mathrm{NChow}(k)_R \to \NMix(k;R)^\oplus$ whose pre-composition with $\mathrm{U}_R$ agrees with $U_R$. This proves the first claim. The second claim follows from the combination of Proposition \ref{prop:computation} with the explicit construction of the category $\Hmo_0(k)_R$.
\end{proof}
\subsection{Relation with noncommutative Artin motives}\label{sub:rel-NCArtin}
As explained in \cite[Thm.~1.2]{Artin}, the category of noncommutative Artin motives $\NAM(k)_R$ can be identified with the smallest additive idempotent complete full subcategory of $\mathrm{NChow}(k)_R$ containing the objects $\mathrm{U}_R(l)$, where $l/k$ is a finite separable field extension. Therefore, making use of the above Proposition \ref{prop:base-change}, we obtain the following result:
\begin{proposition}\label{prop:rel-Artin}
When $R$ is a $\bbQ$-algebra, the category of noncommutative Artin motives $\NAM(k)_R$ identifies with the smallest additive idempotent complete full subcategory of $\NMAM(k;R)$ containing the objects $U_R(l)$, where $l/k$ is a finite separable field extension.
\end{proposition}
\subsection{Some variants}\label{rk:variants}

The category $\Mot(k)$ admits a localizing variant $\Mot_L(k)$ and an $\bbA^1$-homotopy variant $\Mot_{\bbA^1}(k)$; consult \cite{Duke,A1-homotopy} for details. The constructions and results of \S\ref{sub:coefficients}-\ref{sub:relation}, suitably modified, also hold for these variants. Consequently, we obtain triangulated categories $\NMix_L(k;R)$ and $\NMix_{\bbA^1}(k;R)$.
\section{Proof of Theorem \ref{thm:main0}}\label{sec:proof-thm0}
Let us denote by $\mathrm{Sm}(k)$ the category of smooth $k$-schemes.\begin{theorem}\label{thm:comparison-1}
When $R$ is a $\bbQ$-algebra, there exists an $R$-linear symmetric monoidal triangulated functor $\Psi$ making the following diagram commute: 
\begin{equation}\label{eq:diagram-key}
\xymatrix{
\mathrm{Sm}(k)\ar[dd]_-{M_R} \ar[rrr]^-{X \mapsto \perf_\dg(X)} &&& \dgcat(k) \ar[d]^-{U_R} \\
&&& \Ho(\Mot(k;R)) \ar[d]^-{(-)^\vee} \\
\DM(k;R) \ar[rrr]_-{\Psi} &&& \NMix(k;R)^\oplus \subset \Ho(\Mot(k;R))  \,.
}
\end{equation}
The functor $\Psi$ preserves moreover arbitrary direct sums.
\end{theorem}
\begin{remark}
Note that Theorem \ref{thm:main0} follows automatically from Theorem \ref{thm:comparison-1}. Simply restrict yourself to smooth projective $k$-schemes. 
\end{remark}

\begin{proof}
Let $\mathrm{SH}(k)$ be Morel-Voevodsky's stable $\bbA^1$-homotopy category of $(\bbP^1,\infty)$-spectra and ${\bf DA}(k;R)$ the coefficients variant introduced by Ayoub in \cite[\S4]{Ayoub-thesis-2} (see also \cite[\S2.1.1]{Ayoub1}). Recall from {\em loc. cit.} that these categories are related by a symmetric monoidal triangulated functor $(-)_R: \mathrm{SH}(k) \to {\bf DA}(k;R)$. Thanks to the work of R{\"o}dings-Sptizweck-{\O}stv{\ae}r \cite{RSO} and Gepner-Snaith \cite{GS}, the $E^\infty$-ring spectrum $\mathrm{KGL} \in \mathrm{SH}(k)$ which represents homotopy $K$-theory admits a strictly commutative model. Therefore, it gives rise to the symmetric monoidal Quillen model category $\mathrm{Mod}(\mathrm{KGL}_R)$ of $\mathrm{KGL}_R$-modules. The proof of \cite[Cor.~2.5(ii)]{Tabuada-Voevodsky}, with $\mathrm{SH}(k)$ replaced by ${\bf DA}(k;R)$, $\mathrm{KGL}_\bbQ$ replaced by $\mathrm{KGL}_R$, and $\Ho(\Mot(k))_\bbQ$ replaced by $\Ho(\Mot(k;R))$, shows that there exists an $R$-linear fully-faithful symmetric monoidal triangulated functor $\Phi$ making the following diagram commute:
\begin{equation}\label{eq:diagram-1}
\xymatrix{
\mathrm{Sm}(k) \ar[d]_-{\Sigma^\infty(-_+)_R} \ar[rrr]^-{X \mapsto \perf_\dg(X)} &&& \dgcat(k) \ar[d]^-{U_R} \\
{\bf DA}(k;R) \ar[d]_-{-\wedge \mathrm{KGL}_R} &&& \Ho(\Mot(k;R)) \ar[d]^-{(-)^\vee} \\
\Ho(\mathrm{Mod}(\mathrm{KGL}_R)) \ar[rrr]_-{\Phi} &&& \NMix(k;R)^\oplus \subset \Ho(\Mot(k;R))  \,.
}
\end{equation}
The functor $\Phi$ preserves moreover arbitrary direct sums. Let $\mathrm{HZ} \in \mathrm{SH}(k)$ be the $E^\infty$-ring spectrum which represents motivic cohomology. Thanks to the work of Bloch \cite{Bloch}, reformulated by Riou in \cite[\S6]{Riou}, we have $\mathrm{KGL}_R \simeq \bigoplus_{i \in \bbZ} \mathrm{HZ}_R(i)[2i]$. On the other hand, thanks to the work of R{\"o}dings-{\O}stv{\ae}r \cite{RO}, $\mathrm{DM}(k;R)$ identifies with the homotopy category $\Ho(\mathrm{Mod}(\mathrm{HZ}_R))$. As a consequence, base-change along $\mathrm{HZ}_R \to \mathrm{KGL}_R$ gives rise to an $R$-linear symmetric monoidal triangulated base-change functor $-\wedge_{\mathrm{HZ}_R}\mathrm{KGL}_R$ making the following diagram~commute:
\begin{eqnarray}\label{eq:diagram-2}
\xymatrix{
 \mathrm{Sm}(k) \ar[dd]_-{M_R} \ar@{=}[rr] && \mathrm{Sm}(k) \ar[d]^-{\Sigma^\infty(-_+)_R} \\
& & {\bf DA}(k;R) \ar[d]^-{-\wedge \mathrm{KGL}_R} \\
\DM(k;R) \ar[rr]_-{-\wedge_{\mathrm{HZ}_R}\mathrm{KGL}_R} & & \Ho(\mathrm{Mod}(\mathrm{KGL}_R))\,.
}
\end{eqnarray}
The base-change functor $-\wedge_{\mathrm{HZ}_R}\mathrm{KGL}_R$ preserves arbitrary direct sums. Therefore, the desired functor $\Psi$ (resp. diagram \eqref{eq:diagram-key}) is obtained by composing (resp. concatenating) the base-change functor with $\Phi$ (resp. \eqref{eq:diagram-2} with \eqref{eq:diagram-1}).
\end{proof}
Let us denote by $R(1)[2] \in \DM_{\mathrm{gm}}(k;R)\subset \DM(k;R)$ the Tate motive.
\begin{theorem}\label{thm:orbit}
The functor $\Psi$ gives rise to an $R$-linear symmetric monoidal functor
\begin{equation}\label{eq:induced-arbitrary}
\overline{\Psi}: \DM(k;R)/_{\!\!-\otimes R(1)[2]} \too \NMix(k;R)^\oplus\,.
\end{equation}
Moreover, the restriction of \eqref{eq:induced-arbitrary} to $\DM_{\mathrm{gm}}(k;R)/_{\!\!-\otimes R(1)[2]}$ is fully faithful.
\end{theorem}
\begin{proof}
The proof is similar to the one of \cite[Thm.~2.8]{Tabuada-Voevodsky}. Simply replace the categories $\mathrm{DM}(k)_\bbQ$ and $\NMix(k)_\bbQ$ by $\DM(k;R)$ and $\NMix(k;R)$, respectively.
\end{proof}
\section{Proof of Proposition \ref{prop:comparison}}
In order to simplify the exposition we will omit the underscripts of $\Hom$. It will be clear from the context which category we are considering. Let $X:=\mathrm{Spec}(l)$ and $Y:=\mathrm{Spec}(l')$ be two finite {\'e}tale $k$-schemes. We start by showing that the functor $\Psi:\mathrm{NAM}(k;R)\to \NMAM(k;R)$ induces isomorphisms
\begin{equation}\label{eq:induced-1}
\Hom(M_R(X),M_R(Y)[-n]) \stackrel{\sim}{\too} \Hom(\Psi(M_R(X)), \Psi(M_R(Y))[-n])\,.
\end{equation}
The left-hand side identifies with $CH^0(X\times Y)_R$ when $n=0$ and is zero otherwise. In what concerns the right-hand side, we have the following isomorphisms
\begin{eqnarray}
 &  & \simeq \Hom(U_R(\perf_\dg(X)), U_R(\perf_\dg(Y))[-n])\simeq \Hom(U_R(l),U_R(l')[-n]) \label{eq:2} \\
 &  & \simeq K_n(X\times Y)_R\simeq K_n(l\otimes_k l')\,, \label{eq:3}
\end{eqnarray}
where \eqref{eq:2} follows from Remark \ref{rk:last} and \eqref{eq:3} from \eqref{eq:R-mod-isos}. By construction of $\Psi$, the above morphism \eqref{eq:induced-1} with $n=0$ corresponds to the canonical isomorphism $CH^0(X \times Y)_R \simeq K_0(X\times Y)_R$. It remains then to show that the right-hand side of \eqref{eq:induced-1} vanishes when $n \neq 0$. Since by assumption $k$ is finite, the finite dimensional {\'e}tale $k$-algebra $l \otimes_k l'$ is also finite. Thanks to Kuku (see \cite[IV \S1 Prop.~1.16]{Weibel-book}), this implies that all the higher algebraic $K$-theory groups $K_n(l\otimes_k l') \simeq K_n(X \times Y), n \geq 1$, are finite and hence torsion. Using the fact that $R$ is a $\bbQ$-algebra, we therefore conclude that the right-hand side of \eqref{eq:induced-1} vanishes when $n \neq 0$.

The proof of Proposition \ref{prop:comparison} follows now from the above isomorphisms \eqref{eq:induced-1} and from the fact that the objects $M_R(\mathrm{Spec}(l'))[-n]$ (resp. $U_R(l')[-n]$), where $l'/k$ is a finite separable field extension and $n \in \bbZ$, form a set of compact generators of the triangulated category $\mathrm{MAM}(k;R)^\oplus$ (rest. $\NMAM(k;R)^\oplus$).
\section{Proof of Proposition \ref{prop:main}}
As explained in \cite[Example 8.9]{CT}, Hochschild homology gives rise to a symmetric monoidal additive invariant $HH^k: \dgcat(k)^\infty \to \cC(k)^\infty$. Since $R$ is a $k$-algebra, we can then consider the following composition
$$ HH^k_R: \dgcat(k)^\infty \stackrel{HH}{\too} \cC(k)^\infty \stackrel{-\otimes_k R}{\too} \cC(R)^\infty\,,$$
which is also a symmetric monoidal additive invariant. Making use of the equivalence of categories \eqref{eq:equivalence-univ}, we hence obtain an induced symmetric monoidal triangulated functor $HH^k_R: \Ho(\Mot(k;R)) \to \cD(R)$. Let us now show that its restriction 
\begin{equation}\label{eq:induced-HH}
HH^k_R: \NMix(k;R)^\oplus \too \cD(R)
\end{equation}
satisfies conditions (i)-(iii) of Ayoub's weak Tannakian formalism (see \cite[page~19]{Ayoub1}):
\begin{itemize}
\item[(i)] The functor \eqref{eq:induced-HH} admits a right adjoint $\underline{HH}^k_R$.
\item[(ii)] The functor \eqref{eq:induced-HH} admits a symmetric monoidal $2$-section $S$. Moreover, the symmetric monoidal triangulated functor $S$ admits a right adjoint.
\item[(iii)] For every $M \in \cD(R)$ and $N \in \NMix(k;R)^\oplus$ the coprojection morphism $\underline{HH}^k_R(M) \otimes N \to \underline{HH}^k_R (M \otimes HH^k_R(N))$ is invertible.
\end{itemize}
As explained in Remark \ref{rk:symmetric}, the triangulated category $\NMix(k;R)^\oplus$ is compactly generated. Therefore, the proof of item (i) follows from the combination of \cite[Thm.~8.4.4]{Neeman} with the fact that \eqref{eq:induced-HH} preserves arbitrary direct sums. In what concerns item (ii), the symmetric monoidal $2$-section $S$ is given by the composition of the equivalence $\cD(R) \simeq \Ho(\mathrm{Mod}(HR))$ with the following functor 
\begin{eqnarray*}
\Ho(\mathrm{Mod}(HR)) \to \NMix(k;R)^\oplus && M \mapsto (\cA \mapsto \Sigma^\infty(N(w\rep_\dg(\cA,k))_+)\wedge M_{\mathrm{cof}}) \,,
\end{eqnarray*}
where $M_{\mathrm{cof}}$ stands for the a (functorial) cofibrant resolution of $M$. Since $S$ preserves arbitrary direct sums and the category $\cD(R)$ is compactly generated, we conclude once again from \cite[Thm.~8.4.4]{Neeman} that the symmetric monoidal $2$-section $S$ admits a right adjoint. Finally, let us prove condition (iii). By construction, the bifunctors $-\otimes-$ of the categories $\cD(R)$ and $\NMix(k;R)^\oplus$ preserve arbitrary direct sums in each variable. Since the functor \eqref{eq:induced-HH} preserves arbitrary direct sums, we hence conclude from Lemma \ref{lem:sums} below that it suffices to show condition (iii) in the particular case where $N$ belongs to $\{U_R(\cA)\,|\, \cA \in \dgcat_{\mathrm{sp}}(k)\}$. Since these objects are strongly dualizable (see Remark \ref{rk:dualizable}) the proof follows now from \cite[Lem.~2.8]{Ayoub1}.
\begin{lemma}\label{lem:sums}
The right adjoint functor $\underline{HH}^k_R$ preserves arbitrary direct sums.
\end{lemma}
\begin{proof}
The strongly dualizable objects and the compact object of the category $\cD(R)$ are the same. Therefore, since the functor \eqref{eq:induced-HH} is symmetric monoidal and the compact generators of $\NMix(k)_R$ are strongly dualizable, we conclude that \eqref{eq:induced-HH} also preserves compact objects. The proof follows now from \cite[Lem.~2.1.28]{Ayoub-thesis-1}.
\end{proof}
\section{Proof of Proposition \ref{prop:explicit}}
As explained in Remark \ref{rk:symmetric}, $\{U_R(\cA)\,|\, \cA \in \dgcat_{\mathrm{sp}}(k) \}$ is a set of compact generators of the triangulated category $\NMix(k;R)^\oplus$. Moreover, as explained in the proof of Lemma \ref{lem:sums}, the Hochschild homology functor preserves compact objects. Therefore, the proof follows from the application of \cite[Prop.~1.18 and Rk.~2.11]{Pridham} to Hochschild homology and to the differential graded enhancement described in~\S\ref{sub:cubes}.
\section{Proof of Theorem \ref{thm:main1}}
The following subsection, which will play a key role in the proof of Theorem \ref{thm:main1}, is of independent interest:
\subsection{Hopf dg algebras of orbit categories}\label{sub:orbit1}
Let $\cC$ be a triangulated symmetric monoidal category and $O \in \cC$ a $\otimes$-invertible object. Recall from \S\ref{sub:orbit} the construction of the orbit category $\cC/_{\!\!-\otimes O}$ and of the projection functor $\pi: \cC \to \cC/_{\!\!-\otimes O}$. Let $F: \cC \to \cC'$ and $\omega': \cC' \to \cD(R)$ be triangulated symmetric monoidal functors such that $\omega'$ and $\omega:= \omega' \circ F$ satisfy the conditions of Ayoub's weak Tannakian formalism. As usual, we will write $\cH(\cC')$ and $\cH(\cC)$ for the associated Hopf dg algebras.
\begin{proposition}\label{prop:orbit}
Assume the following:
\begin{itemize}
\item[(i)] The functor $\omega'$ preserves arbitrary direct sums.
\item[(ii)] The functor $F$ preserves arbitrary direct sums and admits a right adjoint $G$ which also preserves arbitrary direct sums.
\item[(iii)] The category $\cC$ admits a set $\cG$ of compact generators such that $F(\cG)$ is a set of (compact) generators of~$\cC'$.
\item[(iv)] The functor $F$ factors through $\pi$ and gives rise to isomorphisms 
\begin{eqnarray*}
\Hom_{\cC/_{\!\!-\otimes O}}(\pi(a),\pi(b)) \stackrel{\sim}{\too} \Hom_{\cC'}(F(a),F(b)) && a, b \in \cG\,.
\end{eqnarray*}
\end{itemize}
Under these assumptions, we obtain an isomorphism of Hopf dg algebras 
\begin{equation}\label{eq:isom-Laurent}
\cH(\cC') \simeq \cH(\cC) \otimes_{R[t,t^{-1}]}R\,,
\end{equation} 
where the $R[t,t^{-1}]$-linear structure is induced by the $\otimes$-invertible object $O$.
\end{proposition}
\begin{remark}
Intuitively speaking, the above assumptions (iii)-(iv) expresses $\cC'$ as the ``triangulated envelope'' of the orbit category $\cC/_{\!\!-\otimes O}$.
\end{remark}
\begin{proof}
Let $\underline{\omega}'$ (resp. $\underline{\omega}$) be the right adjoint of $\omega'$ (resp. $\omega$). Under these notations, we have the following isomorphisms
$$(\omega \circ \underline{\omega})(R) =  \omega(\underline{\omega}(R)) \simeq
\omega'((F\circ G)(\underline{\omega}'(R))) \simeq \omega'(\oplus_{i \in \bbZ} \underline{\omega}'(R))\simeq \oplus_{i \in \bbZ} (\omega' \circ \underline{\omega}')(R)\,,
$$
where the first one follows from the natural isomorphisms $\underline{\omega}\simeq G \circ \underline{\omega}'$ and $\omega \simeq \omega' \circ F$, the second one from Lemma \ref{lem:sums-last} below (with $c:=\underline{\omega}'(R)$), and the third one from assumption (i). Consequently, $\cH(\cC)$ identifies with $\cH(\cC')[t,t^{-1}]$ and the functor $F$ gives rise to the following short exact sequence of Hopf $R$-algebras
$$ R[t,t^{-1}] \too  \cH(\cC')[t,t^{-1}] \stackrel{t=0}{\too} \cH(\cC')\,,$$
where the first map is induced by the unit of $\cH(\cC')$. In particular, we obtain the above isomorphism \eqref{eq:isom-Laurent}.
\end{proof}
\begin{lemma}\label{lem:sums-last}
For every object $c \in \cC'$, we have an isomorphism $(F \circ G)(c) \simeq \oplus_{i \in \bbZ} c$. Moreover, if $c$ is a (commutative unital) algebra, then $(F \circ G)(c) \simeq c[t,t^{-1}]$.
\end{lemma}
\begin{proof}
Thanks to assumptions (ii)-(iii), it suffices to show that $(F\circ G)(F(b))$ identifies with $\oplus_{i \in \bbZ} F(b)$ for every $b \in \cG$. By combining assumptions (iii)-(iv) with the definition of the orbit category, we observe that 
$$ \Hom_\cC(a,\oplus_{i \in \bbZ} (b\otimes O^{\otimes i})) \simeq \Hom_{\cC/_{\!\!-\otimes O}}(\pi(a),\pi(b)) \simeq  \Hom_{\cC'}(F(a),F(b))$$
for every $a \in \cG$. This implies that $G(F(b)) \simeq \oplus_{i \in \bbZ} (b\otimes O^{\otimes i})$. Using the $2$-isomorphism $\pi\circ (-\otimes O) \stackrel{\sim}{\Rightarrow} \pi$, and the fact that $F$ preserves arbitrary direct sums and factors through $\pi$, we hence conclude that 
$$(F\circ G)(F(b)) \simeq F(\oplus_{i \in \bbZ}(b \otimes O^{\otimes i})) \simeq \oplus_{i \in \bbZ} (F(b) \otimes F(O)^{\otimes i}) \simeq \oplus_{i \in \bbZ} F(b)\,.$$
This achieves the proof of both claims.
\end{proof}
\subsection*{Proof of Theorem \ref{thm:main1}}
The proof of the first claim is similar to the proof of Proposition \ref{prop:main}; simply replace the category $\NMix(k;R)^\oplus$ by the categories $\NMAM(k;R)^\oplus$ and $\mathrm{MATM}(k;R)^\oplus$ and respective functors. In what concerns the second claim, it suffices to verify the assumptions (i)-(iv) of the above Proposition \ref{prop:orbit} with $\cC:=\mathrm{MATM}(k;R)^\oplus$, $O:=R(1)[2]$, $\cC':=\mathrm{NMAM}(k;R)^\oplus$, $F:=\Psi$, and $\omega:=HH^k_R$. Assumption (i) follows from the construction of the functor $HH^k_R$. The triangulated category $\mathrm{MATM}(k;R)^\oplus$ is compactly generated and all its compact objects are strongly dualizable. Therefore, assumption (ii)  follows from the combination of \cite[Thm.~8.4.4]{Neeman}\cite[Lem.~2.1.28]{Ayoub-thesis-1} with the fact that the functor $\Psi$ is symmetric monoidal and preserves arbitrary direct sums. In what concerns assumption (iii), take for $\cG$ the set of objects $M_R(\mathrm{Spec}(l))(n)[m]$, where $l/k$ is a finite separable field extension and $n,m \in \bbZ$. Note that the corresponding objects 
$$ \Psi(M_R(\mathrm{Spec}(l))(n)[m]) \simeq U_R(l)[m-2n]$$
form a set of compact generators of the triangulated category $\NMAM(k;R)^\oplus$. Finally, assumption (iv) follows from the above Theorem \ref{thm:orbit}.

\begin{remark}[Mixed Tate motives]\label{rk:Tate}
Let $\mathrm{MTM}(k;R)$ be the classical triangulated category of mixed Tate motives, $\mathrm{MTM}(k;R)^\oplus$ the smallest triangulated subcategory of $\DM(k;R)$ which contains  $\mathrm{MTM}(k;R)$ and is closed under arbitrary direct sums, and $\langle U_R(k)\rangle^\oplus$ the smallest triangulated subcategory of $\mathrm{NMix}(k;R)^\oplus$ which contains the $\otimes$-unit $U_R(k)$ and is closed under arbitrary direct sums. Similarly to Theorem \ref{thm:main1}, we have an isomorphism $\cH(\langle U_R(k)\rangle^\oplus)\simeq \cH(\mathrm{MTM}(k;R)^\oplus) \otimes_{R[t,t^{-1}]}R$.

%
\end{remark}
\begin{remark}
Let us denote by $\mathrm{NDM}(k;R)^\oplus$ the smallest triangulated subcategory of $\NMix(k;R)^\oplus$ which contains the objects $U_R(\perf_\dg(X))$, where $X \in \mathrm{SmProj}(k)$, and is closed under arbitrary direct sums. Similarly to Theorem \ref{thm:main1}, we have an isomorphism $ \cH(\mathrm{NDM}(k;R)^\oplus)\simeq \cH(\DM(k;R))\otimes_{R[t,t^{-1}]}R$. 
%
\end{remark}
\section{Proof of Theorem \ref{thm:new}}\label{sec:proof-new}
Recall from Voevodsky \cite[\S3.2]{Voev} that the classical triangulated category of mixed Artin motives $\MAM(k;R)$ is defined as the smallest thick triangulated subcategory of $\DM_{\mathrm{gm}}(k;R)$ containing the objects $M_R(\mathrm{Spec}(l))$, where $l/k$ is a finite separable field extension. As mentioned in \S\ref{sec:statements}, $\MAM(k;R)$ is equivalent to the bounded derived category $\cD^b(\AM(k)_R)$. Since $\AM(k)_R$ identifies with the category of finite dimensional $R$-linear continuous $\mathrm{Gal}(\overline{k}/k)$-representations $\mathrm{rep}_R(\mathrm{Gal}(\overline{k}/k))$, we hence conclude that $\MAM(k;R) \simeq \cD^b(\rep_R(\mathrm{Gal}(\overline{k}/k)))$.
\begin{notation}
Given a finite Galois field extension $\overline{k}/l/k$, let $\MAM(l/k;R)^\oplus$ be the smallest triangulated subcategory of $\DM(k;R)$ which  is closed under arbitrary direct sums and contains the objects $M_R(\mathrm{Spec}(l'))$, where $l/l'/k$ is an intermediate separable field extension. Note that $\MAM(k;R)^\oplus \simeq \bigcup_{\overline{k}/l/k} \MAM(l/k;R)^\oplus$. 
\end{notation}
The category $\MAM(k;R)^\oplus$ is equivalent to the (unbounded) derived category $\cD(\mathrm{Rep}_R(\mathrm{Gal}(\overline{k}/k)))$, where $\mathrm{Rep}_R(\mathrm{Gal}(\overline{k}/k))$ stands for the category of $R$-linear continuous $\mathrm{Gal}(\overline{k}/k)$-representations. In the same vein, $\MAM(l/k;R)^\oplus$ is equivalent to the (unbounded) derived category $\cD(\mathrm{Rep}_R(\mathrm{Gal}(l/k)))$.

In order to simplify the exposition, let us denote by $HH'^k_R$ the composed functor \eqref{eq:functor-Mix-Artin}. The same proof of Proposition \ref{prop:main} shows that the functor $HH'^k_R$, as well as its pre-composition with the inclusions $I_{l/k}: \MAM(l/k;R)^\oplus \subset \MAM(k;R)^\oplus$, satisfies the conditions of Ayoub's weak Tannakian formalism. Consequently, the inclusions $I_{l/k}$ give rise to an induced morphism of Hopf dg algebras (see \cite[Prop.~1.48]{Ayoub1}):
\begin{equation}\label{eq:induced1}
\mathrm{colim}_{\overline{k}/l/k} \cH(\MAM(l/k;R)^\oplus) \too \cH(\MAM(k;R)^\oplus)\,.
\end{equation}
\begin{proposition}\label{prop:Hopf1}
The above morphism \eqref{eq:induced1} is invertible.
\end{proposition}
\begin{proof}
Consider the following adjunctions:
$$
\xymatrix{
\MAM(k;R)^\oplus \ar@<1ex>[d]^-{\underline{I}_{l/k}} && \cD(R) \ar@<1ex>[d]^-{\underline{HH}'^k_R} \\
\MAM(l/k;R)^\oplus \ar@<1ex>[u]^-{I_{l/k}} && \MAM(k;R)^\oplus \ar@<1ex>[u]^-{HH'^k_R}\,,
}
$$
where the existence of the right adjoint $\underline{I}_{l/k}$ follows from the fact that $I_{l/k}$ preserves arbitrary direct sums and that the triangulated category $\MAM(l/k;R)^\oplus$ is compactly generated. Let us write $HH'^{l/k}_R$ (resp. $\underline{HH}'^{l/k}_R$) for the composition $HH'^k_R \circ I_{l/k}$ (resp. $\underline{I}_{l/k} \circ \underline{HH}'^k_R$). Under these notations, we have the isomorphisms:
\begin{eqnarray}
(HH'^k_R \circ \underline{HH}'^k_R)(R) & = & HH'^k_R(\underline{HH}'^k_R(R)) \nonumber \\
& \simeq  &  HH'^k_R(\mathrm{colim}_{\overline{k}/l/k}(I_{l/k} \circ \underline{I}_{l/k})(\underline{HH}'^k_R(R))) \label{eq:isom-11} \\
& \simeq & \mathrm{colim}_{\overline{k}/l/k} HH'^k_R(I_{l/k}(\underline{I}_{l/k}(\underline{HH}'^k_R(R)))) \label{eq:isom-22} \\
&  \simeq & \mathrm{colim}_{\overline{k}/l/k} (HH'^{l/k}_R \circ \underline{HH}'^{l/k}_R)(R)\,,\nonumber
\end{eqnarray}
where \eqref{eq:isom-11} follows from Lemma \ref{lem:aux-1} below (with $N:=\underline{HH'}^k_R(R)$) and \eqref{eq:isom-22} from the fact that $HH'^k_R$ preserves (homotopy) colimits.
\end{proof}
\begin{lemma}\label{lem:aux-1}
For every object $N \in \MAM(k;R)^\oplus$, we have an induced isomorphism $\mathrm{colim}_{\overline{k}/l/k} (I_{l/k} \circ \underline{I}_{l/k})(N) \simeq N$.
\end{lemma}
\begin{proof}
Since the inclusion functor $I_{l/k}$ preserves arbitrary direct sums and the triangulated category $\MAM(l/k;R)^\oplus$ is compactly generated, the right adjoint functor $\underline{I}_{l/k}$ also preserves arbitrary direct sums; see \cite[Lem.~2.1.28]{Ayoub-thesis-1}. Consequently, it suffices to establish the above isomorphism in the particular case where $N=M_R(\mathrm{Spec}(l'))$ with $\overline{k}/l'/k$ an intermediate finite separable field extension. Choose an intermediate finite Galois field extension $\overline{k}/l/k$ containing $l'$. Since the object $M_R(\mathrm{Spec}(l'))$ belongs to $\MAM(l/k;R)^\oplus$, it identifies with $(I_{l/k} \circ \underline{I}_{l/k})(M_R(\mathrm{Spec}(l')))$. This achieves the proof.
\end{proof}
\begin{notation}
Given an intermediate finite Galois field extension $\overline{k}/l/k$, let us write $\cC(\mathrm{Gal}(l/k),R)$ for the Hopf $R$-algebra of functions with multiplication (resp. comultiplication) induced by the multiplication in $R$ (resp. in $\mathrm{Gal}(l/k)$).
\end{notation}
Note that since the Galois group $\mathrm{Gal}(l/k)$ is finite, $\cD(\mathrm{Rep}_R(\mathrm{Gal}(l/k)))$ is equivalent to the derived category of comodules $\cD(\mathrm{coMod}(\cC(\mathrm{Gal}(l/k),R)))$. Note also that $\cC^0(\mathrm{Gal}(\overline{k}/k),R)$ identifies with the colimit $\mathrm{colim}_{\overline{k}/l/k} \cC(\mathrm{Gal}(l/k),R)$. 
\begin{proposition}\label{prop:Hopf2}
When $k \subset \bbC$ and $R$ is a $\bbC$-algebra, the associated Hopf dg algebra $\cH(\MAM(l/k;R)^\oplus)$ is concentrated in degree zero and agrees with $\cC(\mathrm{Gal}(l/k),R)$.
\end{proposition}
\begin{proof}
Consider the following $R$-linear triangulated symmetric monoidal functor
\begin{equation}\label{eq:composition}
\MAM(l/k;R)^\oplus \stackrel{I_{l/k}}{\too} \MAM(k;R)^\oplus \stackrel{\Psi}{\too} \NMAM(k;R)^\oplus \stackrel{HH^k_R}{\too} \cD(R)\,.
\end{equation}
By construction, it preserves arbitrary direct sums. Let $l/l'/k$ be an intermediate separable field extension. Making use of Remark \ref{rk:last} and of the isomorphism $U_R(\perf_\dg(\mathrm{Spec}(l'))) \simeq U_R(l')$, we observe that the image of $M_R(\mathrm{Spec}(l'))$ under the above functor \eqref{eq:composition} identifies with $HH^k_R(l')$. As proved in \cite[\S8]{Homogeneous}, we have the following computation:
$$
HH^k_n(l') \simeq  \left\{ \begin{array}{ll}
         k^{[l':k]}& \mathrm{when}\,\, n=0 \\
         0 & \mathrm{otherwise}\,. \end{array} \right.
         $$
         Therefore, \eqref{eq:composition} identifies with the functor $M_R(\mathrm{Spec}(l')) \mapsto H^\ast_{dR}(\mathrm{Spec}(l'))\otimes_kR$, where $H^\ast_{dR}$ stands for de Rham cohomology. Since by assumption $k \subset \bbC$ and $R$ is a $\bbC$-algebra, the latter functor identifies also with $M_R(\mathrm{Spec}(l')) \mapsto H^\ast_B(\mathrm{Spec}(l'))\otimes_\bbQ R$, where $H^\ast_B$ stands for Betti cohomology; see Grothendieck \cite{Grothendieck}. Hence, under the equivalence of categories between $\MAM(l/k;R)^\oplus$ and $\cD(\mathrm{coMod}(\cC(\mathrm{Gal}(l/k),R)))$, the above functor \eqref{eq:composition} corresponds to the forgetful functor
\begin{equation}\label{eq:forget}
\cD(\mathrm{coMod}(\cC(\mathrm{Gal}(l/k),R))\stackrel{\mathrm{forget}}{\too} \mathrm{Mod}(R))\,;
\end{equation}
see Deligne-Milne \cite[Rk.~6.18]{Deligne-Milne}. The preceding functor \eqref{eq:forget} clearly satisfies the conditions of Ayoub's weak Tannakian formalism. Using the universal property \cite[Prop.~1.55]{Ayoub1}, we conclude finally that the associated Hopf dg algebra is concentrated in degree zero and agrees with $\cC(\mathrm{Gal}(l/k),R)$. This achieves the proof.
\end{proof}
By combining the above Propositions \ref{prop:Hopf1} and \ref{prop:Hopf2}, we obtain the identifications: 
\begin{equation}\label{eq:key}
\cH(\MAM(k;R)^\oplus) \simeq \mathrm{colim}_{\overline{k}/l/k} \cC(\mathrm{Gal}(l/k),R) \simeq \cC^0(\mathrm{Gal}(\overline{k}/k),R)\,.
\end{equation}
This concludes the proof of Theorem \ref{thm:new}.
\section{Proof of Theorem \ref{thm:main2}}
As in the proof of Proposition \ref{prop:main}, the functor $HH^{\overline{k}}_R:\NMix(\overline{k};R)^\oplus \to \cD(R)$, as well as its pre-composition with the following base-change functors (see \S\ref{sub:base-change})
\begin{eqnarray}\label{eq:base-change-last}
-\otimes_l \overline{k}: \NMix(l;R)^\oplus \too \NMix(\overline{k};R)^\oplus && \overline{k}/l/k\,\,\mathrm{finite}\,\,\mathrm{Galois}\,,
\end{eqnarray}
satisfies the conditions of Ayoub's weak Tannakian formalism. Consequently, the functors \eqref{eq:base-change-last} give rise to an induced morphism of Hopf dg algebras: 
\begin{equation}\label{eq:colimit-last}
\mathrm{colim}_{\overline{k}/l/k} \cH(\NMix(l;R)^\oplus) \too \cH(\NMix(\overline{k};R)^\oplus)\,.
\end{equation}
\begin{proposition}\label{prop:colim}
The above morphism \eqref{eq:colimit-last} is invertible.
\end{proposition}
\begin{proof}
Consider the following adjunctions
$$
\xymatrix{
\NMix(\overline{k};R)^\oplus \ar@<1ex>[d]^-{\mathrm{Res}_{\overline{k}/l}} && \cD(R) \ar@<1ex>[d]^-{\underline{HH}^{\overline{k}}_R} \\
\NMix(l;R)^\oplus \ar@<1ex>[u]^-{-\otimes_l \overline{k}} && \ar@<1ex>[u]^-{HH^{\overline{k}}_R}\NMix(\overline{k};R)^\oplus\,,
}
$$
where the existence of the right adjoint $\mathrm{Res}_{\overline{k}/l}$ follows from the fact that $-\otimes_l \overline{k}$ preserves arbitrary direct sums and that the triangulated category $\NMix(l;R)^\oplus$ is compactly generated. We have the following isomorphisms
\begin{eqnarray}
(HH^{\overline{k}}_R \circ \underline{HH}^{\overline{k}}_R)(R) & = & HH^{\overline{k}}_R(\underline{HH}^{\overline{k}}_R(R)) \nonumber \\
 & \simeq & HH^{\overline{k}}_R(\mathrm{colim}_{\overline{k}/l/k}(\mathrm{Res}_{\overline{k}/l}(\underline{HH}^{\overline{k}}_R(R))\otimes_l \overline{k})) \label{eq:last} \\
 & \simeq & \mathrm{colim}_{\overline{k}/l/k} HH^{\overline{k}}_R (\mathrm{Res}_{\overline{k}/l}(\underline{HH}^{\overline{k}}_R(R))\otimes_l \overline{k}) \label{eq:last1} \\
 & \simeq & \mathrm{colim}_{\overline{k}/l/k} (HH^l_R \circ \underline{HH}^l_R)(R)\,, \label{eq:last2}
\end{eqnarray}
where \eqref{eq:last} follows from Lemma \ref{lem:colimit-last} below (with $N:=\underline{HH}^{\overline{k}}_R(R)$), \eqref{eq:last1} from the fact that $HH^{\overline{k}}_R$ preserves (homotopy) colimits, and \eqref{eq:last2} from the natural isomorphisms $\mathrm{Res}_{\overline{k}/l}\circ \underline{HH}^{\overline{k}}_R \simeq \underline{HH}^l_R$ and $HH^{\overline{k}}_R \circ (-\otimes_l \overline{k}) \simeq HH^l_R$.
\end{proof}
\begin{lemma}\label{lem:colimit-last}
For every object $N \in \NMix(\overline{k};R)^\oplus$, we have an induced isomorphism
$ \mathrm{colim}_{\overline{k}/l/k}(\mathrm{Res}_{\overline{k}/l}(N) \otimes_l \overline{k})\simeq N$, where the colimit is taken over the intermediate finite Galois field extensions.
\end{lemma}
\begin{proof}
In order to simplify the exposition, we will omit the underscripts of $\Hom$. It will be clear from the context which category we are considering. Since the base-change functor $-\otimes_l \overline{k}$ preserves compact objects and the triangulated category $\NMix(l;R)^\oplus$ is compactly generated, the functor $\mathrm{Res}_{\overline{k}/l}$ preserves arbitrary direct sums; see \cite[Lem.~ 2.1.28]{Ayoub-thesis-1}. Consequently, the functor $\mathrm{colim}_{\overline{k}/l/k} \mathrm{Res}_{\overline{k}/l}(-)\otimes_{l} \overline{k}$ also preserves arbitrary direct sums. Thanks to the Yoneda Lemma and to the fact that $\{U_R(\cA)\,|\, \cA \in \dgcat_{\mathrm{sp}}(\overline{k})\}$ is a set of compact generators of $\NMix(\overline{k};R)^\oplus$, it suffices then to show that the induced morphisms
$$
\Hom(U_R(\cA)[n],\mathrm{colim}_{\overline{k}/l/k} (\mathrm{Res}_{\overline{k}/l}(U_R(\cB))\otimes_{l} \overline{k}))\too \Hom(U_R(\cA)[n], U_R(\cB))
$$
are invertible for every $\cA, \cB \in \dgcat_{\mathrm{sp}}(\overline{k})$ and $n \in \bbZ$. Since $\overline{k}\simeq \mathrm{colim}_{\overline{k}/l/k} l$, there exists an intermediate finite Galois field extension $\overline{k}/l_0/k$ and a smooth proper dg category $\cA_0 \in \dgcat_{\mathrm{sp}}(l_0)$ such that $\cA_0 \otimes_{l_0}\overline{k}$ is Morita equivalent to $\cA$. We can (and will) assume without loss of generality that $l_0=k$. The proof follows now from the following sequence of isomorphisms:
\begin{eqnarray}
&& \Hom(U_R(\cA)[n],\mathrm{colim}_{\overline{k}/l/k}(\mathrm{Res}_{\overline{k}/l}(U_R(\cB))\otimes_{l}\overline{k})) \nonumber \\
&  & \simeq\mathrm{colim}_{\overline{k}/l/k} \Hom(U_R(\cA_0)[n]\otimes_k l\otimes_l \overline{k}, \mathrm{Res}_{\overline{k}/l}(U_R(\cB))\otimes_{l}\overline{k}) \label{eq:isom-1} \\
&  & \simeq \mathrm{colim}_{\overline{k}/l/k} \mathrm{colim}_{\overline{k}/l'/l} \Hom(U_R(\cA_0)[n] \otimes_k l', \mathrm{Res}_{\overline{k}/l}(U_R(\cB)) \otimes_l l') \label{eq:isom-2} \\
&  & \simeq \mathrm{colim}_{\overline{k}/l/k} \Hom(U_R(\cA_0)[n]\otimes_{k} l, \mathrm{Res}_{\overline{k}/l}(U_R(\cB))) \label{eq:isom-3} \\
&  &  \simeq\mathrm{colim}_{\overline{k}/l/k} \Hom(U_R(\cA_0)[n]\otimes_k l\otimes_l \overline{k}, U_R(\cB)) \nonumber \\
&  &  \simeq\Hom(U_R(\cA)[n], U_R(\cB))\,. \nonumber
\end{eqnarray}
Some explanations are in order: \eqref{eq:isom-1} follows from the compactness of the object $U_R(\cA)[n]$ and from the Morita equivalence $\cA_0 \otimes_k \overline{k} \simeq \cA$; \eqref{eq:isom-2} follows from Sublemma \ref{lem:aux-111} below (with $N_1:= U_R(\cA_0)[n] \otimes_k l$ and $N_2:=\mathrm{Res}_{\overline{k}/l}(U_R(\cB))$); and finally \eqref{eq:isom-3} follows from the cofinal functor $\overline{k}/l/k \mapsto (\overline{k}/l/k,\overline{k}/l/l)$.
\end{proof}
\begin{sublemma}\label{lem:aux-111}
Given a field extension $\overline{k}/l$ and objects $N_1,N_2 \in \NMix(l;R)^\oplus$, we have an induced isomorphism
$$ \mathrm{colim}_{\overline{k}/l'/l} \Hom(N_1\otimes_{l} l', N_2\otimes_{l} l') \stackrel{\sim}{\too} \Hom(N_1\otimes_{l} \overline{k}, N_2\otimes_l \overline{k})\,,$$
where the colimit is taken over the intermediate finite Galois field extensions.
\end{sublemma}
\begin{proof}
As in the proof of the preceding Lemma \ref{lem:colimit-last}, it suffices to show that 
$$
\mathrm{colim}_{\overline{k}/l'/l} \Hom(U_R(\cA)[n]\otimes_{l} l', U_R(\cB)\otimes_{l} l') \too \Hom(U_R(\cA)[n] \otimes_l \overline{k}, U_R(\cB)\otimes_{l} \overline{k})
$$
is an isomorphism for every $\cA, \cB \in \dgcat_{\mathrm{sp}}(l)$ and $n \in \bbZ$. Thanks to Proposition \ref{prop:computation}, the left-hand side identifies with $\mathrm{colim}_{\overline{k}/l'/l} K_n((\cA^\op \otimes \cB)\otimes_{l} l')$ and the right-hand side with $K_n((\cA^\op \otimes \cB)\otimes_{l} \overline{k})$. The proof follows now from the isomorphism $\mathrm{colim}_{\overline{k}/l'/l} l' \simeq \overline{k}$ and from the fact that $K_n(-)$ preserves filtered colimits.
\end{proof}
\begin{corollary}\label{cor:last}
The sequence of Hopf dg algebras of Theorem \ref{thm:main2}
\begin{equation*}
\cC^0(\mathrm{Gal}(\overline{k}/k),R) \too \cH(\NMix(k;R)^\oplus) \too \cH(\NMix(\overline{k};R)^\oplus)
\end{equation*}
identifies with the following filtrant colimit of Hopf dg algebras
$$\mathrm{colim}_{\overline{k}/l/k} \left(\cC(\mathrm{Gal}(l/k),R) \too \cH(\NMix(k;R)^\oplus) \too \cH(\NMix(l;R)^\oplus) \right)\,,$$
where the colimit is taken over the intermediate finite Galois field extensions. The first maps are induced by the functor $\Psi: \MAM(l/k;R)^\oplus \to \NMix(k;R)^\oplus$ and the second ones by the base-change functors $-\otimes_k l: \NMix(k;R)^\oplus \to \NMix(l;R)^\oplus$.
\end{corollary}
\begin{proof}
It follows from the combination of Propositions \ref{prop:Hopf1} and \ref{prop:colim} with the above identifications \eqref{eq:key}.
\end{proof}
\begin{proposition}\label{prop:last}
Given an intermediate finite Galois field extension $\overline{k}/l/k$, we have the following short exact sequence of Hopf dg algebras
\begin{equation}\label{eq:seq-last}
1 \too \cC(\mathrm{Gal}(l/k),R) \too \cH(\NMix(k;R)^\oplus) \too \cH(\NMix(l;R)^\oplus) \too 1\,,\end{equation}
where the first map is induced by the functor $\Psi: \MAM(l/k;R)^\oplus \to \NMix(k;R)^\oplus$ and the second one by the base-change functor $-\otimes_k l: \NMix(k;R)^\oplus \to \NMix(l;R)^\oplus$.
\end{proposition}
\begin{proof}
We need to show that the first map is injective, that the second one is surjective, that the composition is given by the counit followed by the unit, and~that 
$$\cH(\NMix(l;R)^\oplus)\simeq \cH(\NMix(k;R)^\oplus) \otimes_{\cC(\mathrm{Gal}(l/k),R)}R\,.$$ 
Given an element $\sigma\in \mathrm{Gal}(l/k)$, consider the following base-change equivalence
\begin{eqnarray}\label{eq:twist}
{}^\sigma(-):\dgcat(l) \stackrel{\simeq}{\too} \dgcat(l) && \cA \mapsto \cA\otimes_{k, \sigma} l\,.
\end{eqnarray}
As explained in \S\ref{sub:base-change}, \eqref{eq:twist} gives rise to a triangulated symmetric monoidal equivalence (which we still denote by ${}^\sigma(-)$) making the following diagram commute:
\begin{equation}\label{eq:diagram-last}
\xymatrix{
\dgcat(l) \ar[d]_-{U_R} \ar[rr]^-{{}^\sigma(-)} && \dgcat(l) \ar[d]^-{U_R} \\
\NMix(l;R)^\oplus \ar[rr]_-{{}^\sigma(-)} && \NMix(l;R)^\oplus \,.
}
\end{equation}
Under these notations, we have the following isomorphisms:
\begin{eqnarray}
(HH^k_R \circ \underline{HH}^k_R)(R) & = & HH^k_R(\underline{HH}^k_R(R)) \nonumber \\
& \simeq & HH^l_R(\mathrm{Res}_{l/k}(\underline{HH}^l_R(R))\otimes_k l) \label{eq:isomorphism-1} \\
& \simeq & HH^l_R(\oplus_{\sigma \in \mathrm{Gal}(l/k)}{}^\sigma \underline{HH}^l_R(R)) \label{eq:isomorphism-2} \\
& \simeq & \oplus_{\sigma \in \mathrm{Gal}(l/k)} HH^l_R({}^\sigma \underline{HH}^l_R(R))\,, \label{eq:isomorphism-2.5}
\end{eqnarray}
where \eqref{eq:isomorphism-1} follows from the natural isomorphisms $\mathrm{Res}_{l/k} \circ \underline{HH}_R^l \simeq \underline{HH}_R^k$ and $HH^l_R \circ (-\otimes_kl) \simeq HH_R^k$,  \eqref{eq:isomorphism-2} from Lemma \ref{lem:Galois} below (with $N:=\underline{HH}^l_R(R)$), and \eqref{eq:isomorphism-2.5} from the fact that $HH^l_R$ preserves direct sums. Now, following \S\ref{sec:proof-new}, let us write $HH'^{l/k}_R$ for the composed functor \eqref{eq:composition} and $\underline{HH}'^{l/k}_R$ for its right adjoint. Under these notations, we have the following isomorphisms:
\begin{eqnarray}
(HH'^{l/k}_R\circ \underline{HH}'^{l/k}_R) & = & HH'^{l/k}_R(\underline{HH}'^{l/k}_R(R)) \nonumber \\
& \simeq & HH'^{l/k}_R(M_R(\mathrm{Spec}(l))) \label{eq:isomorphism-3} \\
&\simeq & HH^l_R(U_R(l)\otimes_k l) \label{eq:isomorphism-3.5} \\
& \simeq & HH^l_R(\oplus_{\sigma \in \mathrm{Gal}(l/k)}{}^\sigma U_R(l)) \label{eq:isomorphism-4} \\
& \simeq & \oplus_{\sigma \in \mathrm{Gal}(l/k)}HH^l_R({}^\sigma U_R(l)) \label{eq:isomorphism-5}\,.
\end{eqnarray}
Some explanations are in order: \eqref{eq:isomorphism-3} follows from Lemma \ref{lem:cosets} below; \eqref{eq:isomorphism-3.5} from Remark \ref{rk:last} and isomorphism $U_R(\perf_\dg(\mathrm{Spec}(l))) \simeq U_R(l)$; \eqref{eq:isomorphism-4} from the combination of the commutative diagrams \eqref{eq:diagram-base-change} and \eqref{eq:diagram-last} with the canonical isomorphism $l\otimes_k l \simeq \prod_{\sigma \in \mathrm{Gal}(l/k)}{}^\sigma l$; and \eqref{eq:isomorphism-5} from the fact that $HH^l_R$ preserves direct sums. Under the above isomorphisms, the above sequence of Hopf dg algebras \eqref{eq:seq-last} identifies with
$$ \oplus_{\sigma \in \mathrm{Gal}(l/k)}HH^l_R({}^\sigma U_R(l)) \to \oplus_{\sigma\in \mathrm{Gal}(l/k)} HH^l_R({}^\sigma \underline{HH}^l_R(R))\to (HH^l_R \circ \underline{HH}^l_R)(R)\,,$$
where the first map is the direct sum (indexed by $\sigma \in \mathrm{Gal}(l/k)$) of the units of the Hopf dg algebras $HH^l_R({}^\sigma \underline{HH}_R^l(R))$ and the second map is the projection onto the factor $1 \in \mathrm{Gal}(l/k)$. The proof follows now automatically from this description.
\end{proof}
\begin{lemma}\label{lem:Galois}
For every object $N \in \NMix(l;R)^\oplus$, we have a canonical isomorphism $\mathrm{Res}_{l/k}(N) \otimes_k l\simeq \oplus_{\sigma \in \mathrm{Gal}(l/k)}{}^\sigma N$.
\end{lemma}
\begin{proof}
Note first that the functor $-\otimes_kl$ admits a right adjoint
\begin{eqnarray}\label{eq:restriction}
(-)_k: \dgcat(l) \too \dgcat(k) && \cA \mapsto \cA_k\,.
\end{eqnarray}
Clearly, the functor \eqref{eq:restriction} preserves Morita equivalences, filtered colimits, and short exact sequences of dg categories. Since the field extension $l/k$ is in particular finite and separable, the functor \eqref{eq:restriction} preserves moreover the smooth proper dg categories; see \cite[Prop.~7.5]{Artin}. Making use of the equivalence \eqref{eq:equivalence-univ0}, we hence obtain a triangulated functor (which we still denote by $(-)_k$) making the left-hand side diagram commute and fitting into the right-hand side adjunction:
\begin{equation}\label{eq:diagrams}
\xymatrix{
\dgcat_{\mathrm{sp}}(l) \ar[d]_-{U_R} \ar[r]^-{(-)_k} & \dgcat_{\mathrm{sp}}(k) \ar[d]^-{U_R} &  & \NMix(l;R)^\oplus \ar@<1ex>[d]^-{(-)_k} \\
\NMix(l;R)^\oplus \ar[r]_-{(-)_k} & \NMix(k;R)^\oplus && \NMix(k;R)^\oplus \ar@<1ex>[u]^-{-\otimes_kl}\,.}
\end{equation}
In particular, $\mathrm{Res}_{l/k}\simeq (-)_k$. Since $(-)_k$ preserves arbitrary direct sums, it suffices then to show that $\mathrm{Res}_{l/k}(U_R(\cA)) \otimes_k l$ is isomorphic to $\oplus_{\sigma \in \mathrm{Gal}(l/k)}{}^\sigma U_R(\cA)$ for every $\cA \in \dgcat_{\mathrm{sp}}(l)$. Consider the following isomorphisms:
\begin{eqnarray}
\mathrm{Res}_{l/k}(U_R(\cA))\otimes_k l & \simeq & U_R(\cA\otimes_k l)_k\otimes_k l \nonumber \\
& \simeq & U_R((\cA\otimes_l l)_k)\otimes_k l \label{eq:isomorphism-11} \\
& \simeq & U_R(\cA\otimes_l (l \otimes_k l)) \label{eq:isomorphism-22} \\
& \simeq & \oplus_{\sigma \in \mathrm{Gal}(l/k)}{}^\sigma U_R(\cA)\,, \label{eq:isomorphism-33}
\end{eqnarray}
where \eqref{eq:isomorphism-11} follows from the commutativity of the left-hand side of \eqref{eq:diagrams}, \eqref{eq:isomorphism-22} from the commutative diagram \eqref{eq:diagram-base-change}, and \eqref{eq:isomorphism-33} from the combination of the isomorphism $l \otimes_k l \simeq \prod_{\sigma \in \mathrm{Gal}(l/k)}{}^\sigma l$ with the commutative diagram \eqref{eq:diagram-last}. This achieves the proof.
\end{proof}
\begin{lemma}\label{lem:cosets}
We have an isomorphism $\underline{HH}'^{l/k}_R(R)\simeq M_R(\mathrm{Spec}(l))$.
\end{lemma}
\begin{proof}
In order to simplify the exposition we will omit the underscripts of $\Hom$. It will be clear from the context which category we are considering. Let $l/l'/k$ be an intermediate separable field extension and $n \in \bbZ$. Since $HH'^{l/k}_R$ preserves arbitrary direct sums, it suffices to show that the $R$-modules
\begin{eqnarray}\label{eq:two-sided}
& & \Hom(M_R(\mathrm{Spec}(l')),M_R(\mathrm{Spec}(l))) \quad \Hom(HH'^{l/k}_R(M_R(\mathrm{Spec}(l'))),R)
\end{eqnarray}
are isomorphic. As explained in the proof of Proposition \ref{prop:Hopf2}, the right-hand side identifies with $\Hom_{\cD(R)}(R^{[l':k]},R)\simeq R^{[l':k]}$. In what concerns the left-hand side, it identifies with $CH^0(\mathrm{Spec}(l') \times \mathrm{Spec}(l))_R \simeq K_0(l' \otimes_k l)_R$. Let us denote by $H'$ (resp. $H$) the closed subgroup of $G:=\mathrm{Gal}(\overline{k}/k)$ such that $l'=\overline{k}^{H'}$ (resp. $l=\overline{k}^H$). Since the subgroup $H$ is normal and $H \subset H'$, Galois theory gives rise to the isomorphisms
$$ l' \otimes_k l \simeq \overline{k}^{H'} \otimes_k \overline{k}^H \simeq \prod_{\overline{\sigma} \in H'\backslash G/H} \overline{k}^{H' \cap (\sigma^{-1} H\sigma)} = \prod_{\overline{\sigma} \in H'\backslash G/H} \overline{k}^{H} \simeq \prod_{\overline{\sigma} \in H'\backslash G} l \simeq  l^{[l':k]}\,,$$
where $H'\backslash G/H$ stands for the set of double cosets. Consequently, we conclude that the left-hand side of \eqref{eq:two-sided} identifies with $K_0(l^{[l':k]})_R \simeq \bbZ^{[l':k]}\otimes_\bbZ R \simeq R^{[l':k]}$. This concludes the proof. 
\end{proof}
We now have all the ingredients necessary to prove Theorem \ref{thm:main2}. When the field extension $\overline{k}/k$ is finite, the proof follows from Proposition \ref{prop:last}. Otherwise, it follows from combination of Proposition \ref{prop:last} with Corollary \ref{cor:last}. 
\begin{remark}
The above proofs of Proposition \ref{prop:main} and Theorems \ref{thm:main1}-\ref{thm:main2} hold {\em mutatis mutandis} with $\NMix(k;R)$ replaced by its localizing variant $\NMix_L(k;R)$ and $\bbA^1$-homotopy variant $\NMix_{\bbA^1}(k;R)$; see Remark \ref{rk:variants}.
\end{remark}

\end{document}